\documentclass[reqno,12pt,letterpaper]{amsart}
\usepackage{amsmath,amssymb,amsthm,graphicx,mathrsfs,url}
 \usepackage[usenames,dvipsnames]{color}
\usepackage[colorlinks=true,linkcolor=Red,citecolor=Green]{hyperref}
\usepackage{amsxtra}

\allowdisplaybreaks

\DeclareGraphicsExtensions{.mps}

\def\?[#1]{\textbf{[#1]}\marginpar{\Large{\textbf{??}}}}

\let\epsilon=\varepsilon 

\newcommand{\eps}{\epsilon}

\newcommand{\mc}{\mathcal}

\newcommand{\cB}{\mc B}

\newcommand{\cU}{\mc U}
\newcommand{\ol}{\overline}

\newcommand{\hra}{\hookrightarrow}
\newcommand{\RR}{{\mathbb R}}
\newcommand{\R}{{\mathbb R}}
\newcommand{\Sph}{\mathbb{S}}
\newcommand{\Ric}{\mathrm{Ric}}
\newcommand{\HH}{{\mathbb H}}

\newcommand{\SP }{{\mathbb S}}
\newcommand{\CC}{{\mathbb C}}
\newcommand{\C}{{\mathbb C}}
\newcommand{\ZZ}{{\mathbb Z}}
\newcommand{\Z}{{\mathbb Z}}
\newcommand{\CI}{{C^\infty}}
\newcommand{\CIc}{{C^\infty_{\rm{c}}}}
\newcommand{\dS}{{\mathrm{dS}}}
\newcommand{\cO}{{\mathcal O}}
\newcommand{\loc}{{\mathrm{loc}}}
\newcommand{\comp}{{\mathrm{comp}}}
\newcommand{\Sl}{{\mathcal{S}\ell}}
\newcommand{\Dl}{{\mathcal{D}\ell}}
\newcommand{\ran}{\operatorname{ran}}
\newcommand{\Id}{\operatorname{Id}}

\newtheorem{thm}{Theorem}
\newtheorem{prop}{Proposition}[section]

\newtheorem{lemm}[prop]{Lemma}

\numberwithin{equation}{section}
\newtheorem{rmk}[prop]{Remark}

\DeclareMathOperator{\Res}{Res}

\let\Im=\Imag

\DeclareMathOperator{\rank}{rank}

\let\Re=\Real

\DeclareMathOperator{\supp}{supp}
\DeclareMathOperator{\vol}{vol}

\DeclareMathOperator{\tr}{tr}

\newcommand{\pa}{\partial}
\newcommand{\la}{\langle}
\newcommand{\ra}{\rangle}

\title{Resonances for obstacles in hyperbolic space}

\author{Peter Hintz}
\email{phintz@berkeley.edu}
\author{Maciej Zworski}
\email{zworski@math.berkeley.edu}
\address{Department of Mathematics, University of California,
Berkeley, CA 94720, USA}

\begin{document}

\begin{abstract}
We consider scattering by star-shaped obstacles in hyperbolic space and show
that resonances satisfy a universal bound $ \Im \lambda \leq - \frac12 $
which is optimal in dimension $ 2 $. In odd dimensions we also show that 
$ \Im \lambda \leq - \frac{\mu}{\rho} $ for a universal constant $\mu$, where $ \rho $ is the radius of a ball containing the obstacle; this gives an improvement for small obstacles. 
In dimensions $3$ and higher the proofs follow the classical vector field approach of
Morawetz, while in dimension $2$ we obtain our bound by working with spaces coming
from general relativity. We also show that in odd dimensions resonances of 
small obstacles are close, in a suitable sense, to Euclidean resonances.
\end{abstract}

\dedicatory{Dedicated to the memory of Cathleen Synge Morawetz 1923--2017}

\maketitle

\section{Introduction}

For $\kappa>0$ we define hyperbolic $n$-space with constant curvature $-\kappa^2$ as
\begin{equation}
\label{EqHyp}
  (\HH^n_\kappa,g_\kappa) = (\R^n, dr^2 + s_\kappa^2 h),
\end{equation}
where $(r,\omega)$ are polar coordinates on $\R^n$, $h=h(\omega,d\omega)$ is the round metric on $\Sph^{n-1}$, and $s_\kappa(r)=\kappa^{-1}\sinh(\kappa r)$. 
We include Euclidean space as the case of $ \kappa = 0 $, 
$ s_0 ( r ) = r $. 

\begin{figure}[!ht]
  \centering
  \includegraphics[width=1.8in]{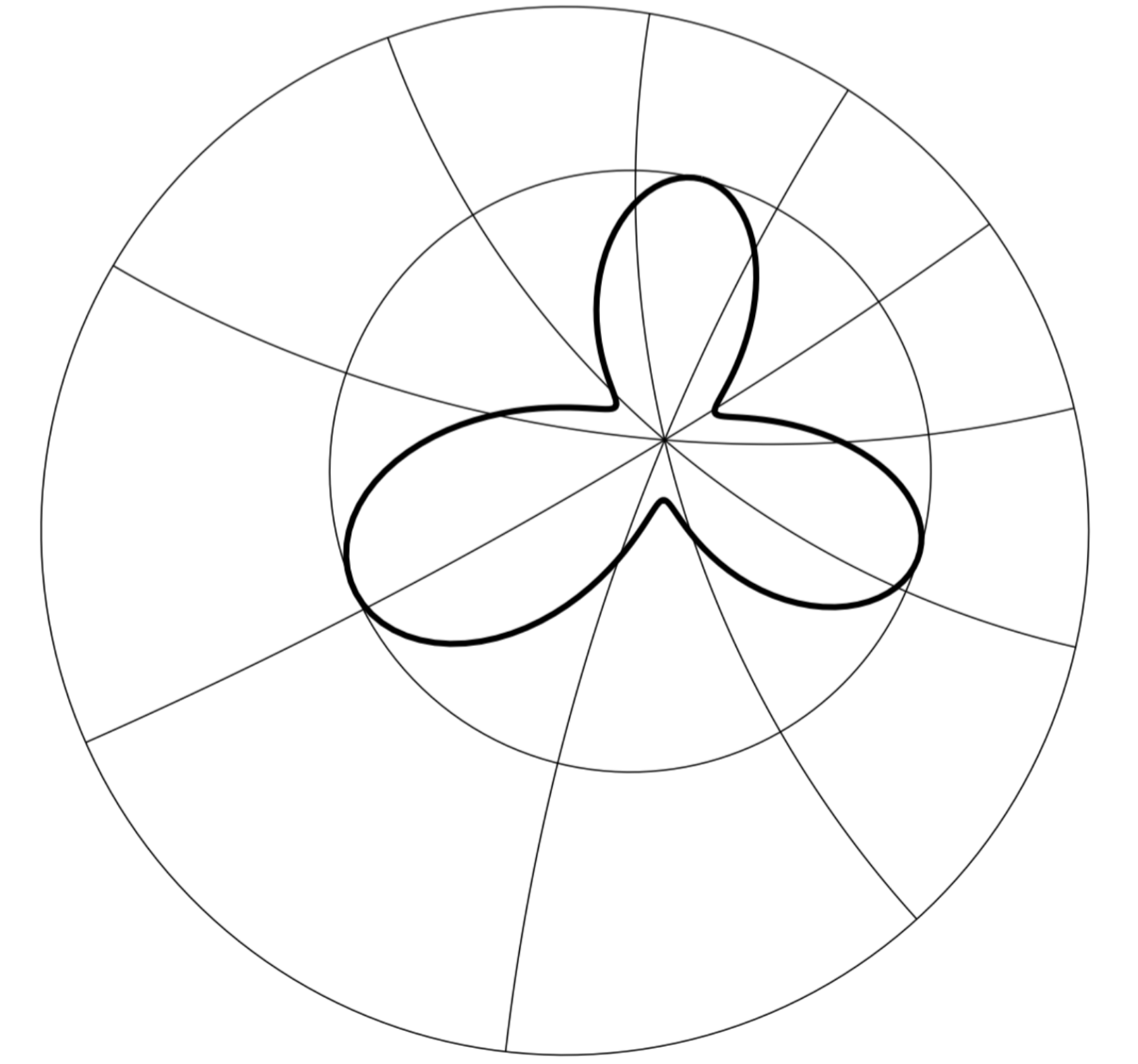} 
  \hspace{0.2in} \includegraphics[width=2.7in]{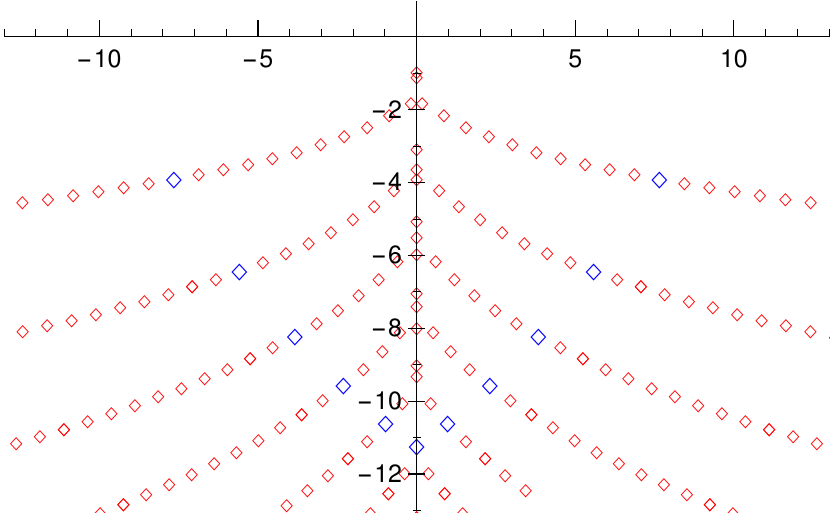}
  \caption{\textit{Left:} a star-shaped obstacle in the Poincar\'e disc with resonances satisfying a
 universal bound $ \Im \lambda \leq - \frac12 $. \textit{Right:} resonances of a disk with radius $R=1$ in $\HH^2$. Highlighted are resonances corresponding to $\ell=12$ (in the notation of \S \ref{s:ball}).}
       \label{f:starh}
\end{figure}

Suppose that $ \mathcal O \subset \RR^n \simeq \HH_\kappa^n $ is a bounded open set with smooth boundary, and denote by
\begin{equation} 
\label{EqHypOp} P_\kappa := - \Delta_{g_\kappa } - \left(\textstyle \frac{n-1}2 \right)^2 \kappa^2
\end{equation}
the self-adjoint operator on $ L^2 (  \mathbb H_\kappa^n \setminus \mathcal O  ,
d \vol_{g_\kappa } ) $ with domain 
$$ \mathcal D ( P_\kappa ) :=  H^2 ( \mathbb H_\kappa^n \setminus \mathcal
O ) \cap H_0^1 (\mathbb H_\kappa^n \setminus \mathcal O ) . $$
The resolvent of $ P_\kappa $, $ \kappa > 0 $, 
\begin{equation}
\label{EqHypResolvent}
  R_\kappa(\lambda) := (P_\kappa-\lambda^2)^{-1} \colon L^2(\HH^n_\kappa\setminus\cO) \to L^2(\HH^n_\kappa\setminus\cO), \ \ \ \Im \lambda > 0 , 
\end{equation}
continues meromorphically to a family of operators defined on $ \CC$:
\[
R_\kappa(\lambda) \colon L^2_{\comp}(\HH^n_\kappa\setminus\cO) \to L^2_{\loc} (\HH^n_\kappa\setminus\cO).
\]
For $\kappa=0$, the same result is true when $ n $ is odd; in even dimensions
the continuation takes place on the logarithmic plane.

We denote the set of poles of $ R_\kappa ( \lambda ) $ 
(included according to their multiplicities \eqref{eq:multk}) 
by $ \Res (\mathcal O, \kappa ) $. The elements of $ \Res (\mathcal O, \kappa ) $
are called {\em scattering resonances}
and they determine decay and oscillations of reflected waves outside of $ \mathcal O $ -- see
\cite{revres} for a recent survey and references.
In the odd-dimensional Euclidean case their study goes back to classical works of Lax--Phillips \cite{LP} and Morawetz \cite{Moa}, and the relation between 
the distribution of resonances and the geometry of obstacles has been much studied, especially for high energies ($ |\Re \lambda| \to \infty $) -- see 
\cite[\S 2.4]{revres}.

When the obstacle is star-shaped, a universal lower bound on {\em resonance widths}, $ |\Im \lambda | $, can be given in terms of the radius of the support of the
obstacle. Following earlier contributions of Morawetz \cite{Moa,Moe,Mo2} and using Lax--Phillips theory \cite{LP}, Ralston \cite{RalstonDecay}
obtained the bound
\begin{equation}
\label{eq:Ralston}  \mathcal O \subset B_{\RR^n }  ( x_0 , \rho) \ \Longrightarrow \ \
\inf_{ \lambda \in \Res (\mathcal O, 0 ) } | \Im \lambda | \geq \rho^{-1}
\end{equation}
for odd  $ n \geq 3 $. Remarkably this bound is optimal in dimensions three and five -- see Fig.~\ref{f:sphere} and \cite{HZ} for a discussion of this result.

In this paper we investigate analogues of \eqref{eq:Ralston} for 
$ \mathcal O \subset B_{\HH^n_\kappa} ( x_0 , \rho ) $. The first result shows that the resonance widths have a universal lower bound independent of the diameter of the obstacle. Intuitively this is due to the fact that infinity is much ``larger" in the hyperbolic case. 

\begin{thm}
\label{t:1}
Suppose that $ \mathcal O \subset \HH^n_\kappa $ is a
star-shaped obstacle. Then 
\begin{equation}
\label{eq:t1} 
\inf_{ \lambda \in \Res (\mathcal O , \kappa  ) } | \Im \lambda | \geq  \kappa / 2  .
\end{equation}
\end{thm}

Here, $\mathcal O$ being star-shaped means that there exists a point $o\in\mathcal O$ so that for all $p\in\partial\mathcal O$, the geodesic segment from $o$ to $p$ is contained in $\overline{\mathcal O}$. The estimate is sharp when $n=2$ and $\mathcal O=\emptyset$.

For the proof, we use the characterization of resonances and resonant states given by Vasy \cite{vasy2,vasy1}, and use ideas from general relativity to prove estimates on resonant states directly (\S \ref{s:rela}). In dimensions $n\geq 3$, we give an alternative argument based on the vector field method of Morawetz (\S \ref{s:moe}) and prove exponential energy decay for solutions of a certain wave equation on $\HH^n_\kappa$ with Dirichlet boundary conditions on $\partial\mathcal O$. (Relating such decay estimates to bounds for resonances via resonance expansions of waves, discussed in Theorem~\ref{t:expansion} below, requires the technical assumption that the boundary $\partial\mathcal O$ be nowhere flat to infinite order: this is sufficient in order for the resolvent to satisfy high energy estimates.) Using a hyperbolic space version of Morawetz's estimate for $ n \geq 3 $ and a slight refinement of the argument from \cite{Moa} gives an improvement for small obstacles in odd dimensions; this is due to the sharp Huyghens principle.

\begin{thm}
\label{t:22}
Suppose that $ \mathcal O \subset \HH^n_\kappa $ is a star-shaped obstacle and that $ n \geq 3 $ is {\em odd}; assume that the boundary $\partial\mathcal O$ is nowhere flat to infinite order. Then
\begin{equation}
\label{eq:t2}
 \mathcal O \subset B_{\HH^n_\kappa }  ( x_0 , \rho) \ \Longrightarrow \ \
\inf_{ \lambda \in \Res ( \mathcal O , \kappa ) } | \Im \lambda | \geq 
\mu \rho^{-1}
\end{equation}
for a universal constant $ \mu $ (see \eqref{eq:unialph} for a more precise 
statement).
\end{thm}

\noindent
{\bf Remark.} Jens Marklof suggested a formulation of Theorems \ref{t:1} and \ref{t:22} which does not depend on $ \kappa $: there exist 
constants $ c_n $ such that for star-shaped obstacles  $\mathcal O \subset \HH^n_{\kappa } $, $ n $ odd,  
\[ \mathcal O \subset B_{\HH^n_\kappa }  ( x_0 , \rho) \ \Longrightarrow \ \
\inf_{ \lambda \in \Res ( \mathcal O , \kappa ) } | \Im \lambda | \geq 
c_n \frac{ \vol (  \partial B_{\HH^n_\kappa }  ( 0 , \rho) ) }{
\vol (  B_{\HH^n_\kappa }  ( 0 , \rho) ) ) } . \]

\begin{figure}[htbp]
\begin{center}
\includegraphics[width=2.7in]{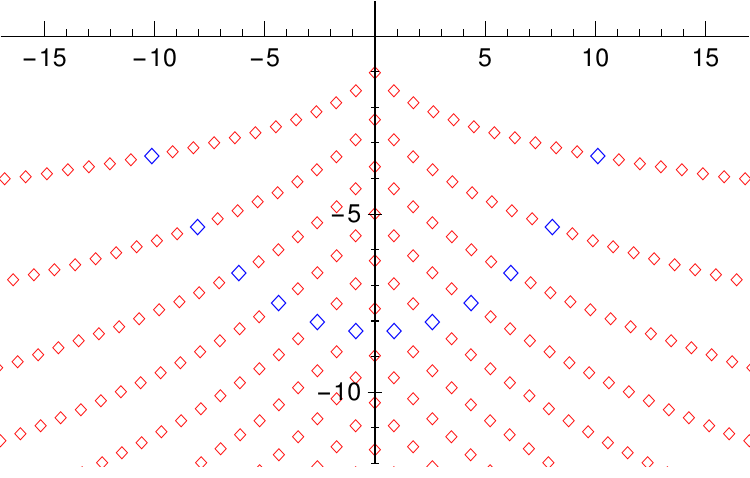}
\hspace{0.2in}
\includegraphics[width=2.7in]{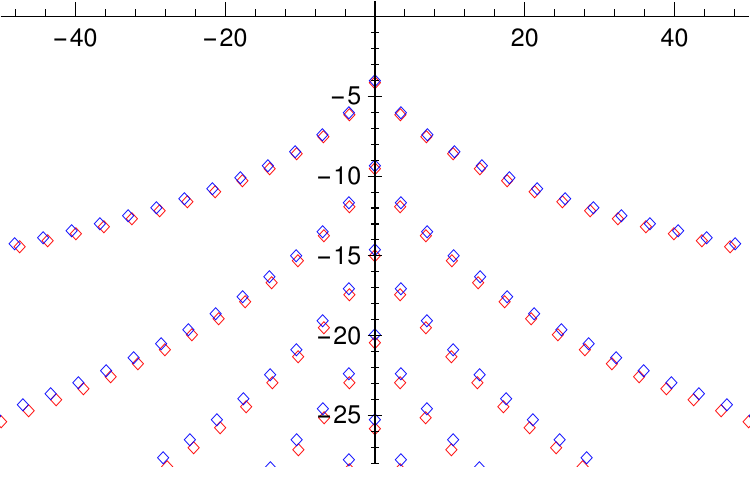}
\caption{\label{f:sphere}
\textit{Left:} resonances for the ball of radius one in 
$ \RR^3 $. For each spherical
momentum $ \ell $ they are given by solutions of  $ H^{(2)}_{\ell +1/2} (\lambda ) = 0 $ where $ H^{(2) }_\nu $ is the Hankel function of the second
kind and order $ \nu $. Each zero appears as a resonance of multiplicity
$ 2 \ell +1 $; highlighted are resonances corresponding to $\ell=12$. \textit{Right:} resonances of the ball with radius $R=0.25$ in $\HH^3$ (red) and in $\RR^3$ (blue); this illustrates Theorem \ref{ThmMain1}.
}
\label{FigHypRes}
\end{center}
\end{figure}

We expect that $ \mu = 1 $ in \eqref{eq:t2}. (An adaptation of Ralston's argument \cite{RalstonDecay} should work 
but would require some buildup of scattering theory; for a proof of his 
crucial estimate without using Lax--Phillips theory, see \cite[Exercise 3.5]{dizzy}.)
That the estimate~\eqref{eq:t2} is independent of $\kappa$ is related to rescaling:
identifying an obstacle with a subset of $ \RR^n $ and 
denoting by $ x \mapsto \epsilon x $ the Euclidean dilation, we see that 
if $\sigma \in \Res ( \epsilon \mathcal O, 1 ) $ 
then $\eps\sigma \in \Res ( \mathcal O, \epsilon ) $,
and $ \epsilon \sigma $ should be close to a resonance in $ \Res ( \mathcal O, 0 ) $.
So even though the bound \eqref{eq:t1} gets worse for small $ \kappa $, the
bound in odd dimensions is close to \eqref{eq:Ralston} and improves for small diameters.
This is illustrated by 
Fig.~\ref{FigHypRes} and confirmed
by the following theorem:

\begin{thm}
\label{ThmMain1}
Suppose that $ \mathcal O \subset \HH_\kappa^n \simeq 
\RR^n $ is an arbitrary bounded obstacle with smooth boundary and that $n\geq 3$ is odd. Then
\[ \Res(\cO,\kappa)\to\Res(\cO,0), \ \ \ \kappa\to 0, \]
locally uniformly and with multiplicities. 
\end{thm}
A more precise version is given in Theorem~\ref{ThmMain} in \S \ref{s:smalle}.

\medskip

\noindent
{\sc Acknowledgments.} We would like to thank Steve Zelditch whose comments on \cite{revres} provided motivation for this project, and Volker Schlue for useful discussions. We would also like to thank an anonymous referee for a careful reading of the manuscript and a number of valuable suggestions which in particular led to the addition of \S\ref{s:nontrap}. PH is grateful to the Miller Institute at the University of California, Berkeley for support, and MZ acknowledges partial support under the National Science Foundation grant DMS-1500852. This research was partially conducted during the period PH served as a Clay Research Fellow.

\section{Resonances for balls in \texorpdfstring{$ \HH_\kappa^n $}{hyperbolic space}}
\label{s:ball}

As motivation for the proofs of the main results we present computations of
resonances of the geodesic ball of radius $R$ in $\HH^n_\kappa$, with Dirichlet boundary conditions. (See also Borthwick \cite{bo:upper}.)

The starting point is the calculation
\[
  s_\kappa^{\frac{n-1}{2}}\Bigl(-\Delta_{\HH^n_\kappa}-\Bigl(\frac{n-1}{2}\Bigr)^2\kappa^2-\sigma^2\Bigr)s_\kappa^{-\frac{n-1}{2}} = D_r^2 + s_\kappa^{-2}\Bigl(-\Delta_{\SP^{n-1}}+\frac{(n-1)(n-3)}{4}\Bigr)-\sigma^2
\]
for $s_\kappa(r)=\kappa^{-1}\sinh(\kappa r)$, see Lemma \ref{l:conj} below. Decomposing into spherical harmonics and using that the eigenvalues of $\SP^{n-1}$ are given by $\ell(\ell+n-2)$, $\ell\in\ZZ_{\geq 0}$, it suffices to study the radial operator
\[
  P_{n,\ell}(\sigma) := D_r^2 + s_\kappa^{-2}\Bigl(\frac{(n-1)(n-3)}{4}+\ell(\ell+n-2)\Bigr)-\sigma^2.
\]
Our objective is to calculate non-trivial solutions of $P_{n,\ell}(\sigma)u=0$ which are \emph{outgoing}, which means that $u=e^{i r\sigma} v(\coth\kappa r)$, where $v=v(x)$, $x=\coth\kappa r$, is smooth in $[1,\infty)_x$ \emph{down to $x=1$}. (This is precisely the condition~\eqref{eq:outgoing} in Theorem~\ref{t:mercont} below, after rescaling to the case $\kappa=1$.) The space of such $u$ is a 1-dimensional space, and if, for fixed $\sigma$, such a $u$ vanishes at $r=R$, then $\sigma$ is a resonance for the $R$-ball in $\HH^n_\kappa$. By direct computation, we have
\[
  P_{n+2,\ell-1}(\sigma) = P_{n,\ell}(\sigma),
\]
hence $P_{n,\ell}=P_{n+2\ell,0}$, and it suffices to calculate outgoing solutions $u$ of $P_{n,0}(\sigma)u=0$ for all $n\geq 2$. Using $(\sinh\kappa r)^{-2}=x^2-1$, one finds that $e^{-i r\sigma}P_{n,0}(\sigma)e^{i r\sigma}v(\coth\kappa r)=0$ is equivalent to
\[
  \tilde P_n(\sigma)v := \Bigl(\pa_x(1-x^2)\pa_x + 2 i\kappa^{-1}\sigma\pa_x + \frac{(n-1)(n-3)}{4}\Bigr)v=0
\]
Changing variables $y=(1-x)/2$, this is a hypergeometric equation. For \emph{odd} $n=2k+1$, smooth solutions of this equation are \emph{polynomials} of $x$. To see this directly, we make the ansatz
\[
  v(x) = \sum_{j=0}^\infty \frac{(x-1)^j}{\Gamma(j+1-i\kappa^{-1}\sigma)}c_{k,j},\quad c_{k,0}=1;
\]
plugged into the ODE, this yields the recursion relation
\[
  c_{k,j}=\frac{k(k-1)-j(j-1)}{2j}c_{k,j-1}, \ \ j\geq 1,
\]
in particular $c_{k,j}=0$ for all $k\geq j$. Therefore, multiplying through by $\Gamma(k-i\kappa^{-1}\sigma)$ in order to deal with integer coincidences $k-i\kappa^{-1}\sigma\in\ZZ_{<0}$, the non-trivial outgoing solution $u_n(r;\sigma)$ of $P_n(\sigma)u_n(r;\sigma)=0$, $n=2k+1$, is given by
\[
  u_{2k+1}(r;\sigma) = e^{i r\sigma}\sum_{j=0}^{k-1}
  \frac{(\coth(\kappa r)-1)^j}{2^j j!} 
  \prod_{l=1}^j\bigl(k(k-1)-l(l-1)\bigr)
  \prod_{m=j+1}^{k-1}(m-i\kappa^{-1}\sigma)
   ,
\]
where the product $ \prod_{ l =1}^0 $ is defined to be $ 1 $.

Note that $e^{-ir\sigma}u_{2k+1}(r;\sigma)$ is a polynomial in $\sigma$ of degree $k-1$. If the size $R$ of the obstacle is fixed, the zeros of $u_{2k+1}(R;\sigma)=0$ are the resonances. See Fig.~\ref{FigHypRes}.

Suppose now the obstacle is large, so $\coth\kappa R$ is close to $1$, and fix $k$. Then $u_{2k+1}(R;\sigma)$, as a function of $\sigma$, is well-approximated by a constant multiple of
\[
  e^{i R\sigma}\prod_{m=1}^{k-1}(m-i\kappa^{-1}\sigma),
\]
whose zeros are located at $-i\kappa m$, $m=1,\ldots,k-1$. By Rouch\'e's theorem, this implies that for $n$ odd, $\kappa>0$, and $\eps>0$ fixed, there exists $R_0>0$ such that for spherical obstacles in $\HH^n_\kappa$ with radius $R>R_0$, there exists a resonance $\sigma$ with $|\sigma+i\kappa|<\eps$. (For comparison, Theorem~\ref{t:1} only gives $\Im\sigma\leq - \kappa/2$.)

One can also numerically compute resonances on even-dimensional hyperbolic spaces -- see Fig.~\ref{f:starh}. When the diameter of a spherical obstacle in $\HH^2$ tends to zero, numerical experiments suggest that the topmost resonance converges to $-i/2$, the topmost resonance for the free resolvent on $\HH^2$.

\section{Preliminaries}

In this section we review the meromorphic continuation of the resolvent on 
asymptotically hyperbolic manifolds with obstacles, resonance free strips 
and resonance expansions in the non-trapping case, and the vector field approach via the stress--energy tensor.

\subsection{Meromorphic continuation of the resolvent}
\label{s:mercont}

Let $ ( M, g )  $ be an (even) asymptotically hyperbolic manifold with boundary. This means that $ M $ admits a compactification to a manifold $ \bar M $ with boundary $ \partial \bar M = \partial M \cup \partial_1\bar M $, where $\partial_1\bar M$ is the conformal boundary of $M$; moreover, the Riemannian metric $g$ is smooth on $M$, while in a collar neighborhood $[0,\eps)_x\times(\partial_1\bar M)_y$ of the conformal boundary, the rescaled metric $\bar g(x,y,dx,dy):=x^2 g(x,y,dx,dy)$ is a smooth Riemannian metric on $\bar M$ whose Taylor expansion $x=0$ contains only even powers of $x$ (see also \cite{Gui05}), and $|dx|_{\bar g}^2=1$ at $\partial_1\bar M$. See \cite[\S5.1]{dizzy} for further discussion.

An example considered in this paper is $ ( M, g) = ( \mathbb H_1^n \setminus \mathcal O , g_1 ) $. We discuss the conformal compactification and its smooth structure explicitly in \S\ref{s:rela}.

The following theorem is essentially due to Vasy \cite{vasy2,vasy1} -- see also \cite{V4D} for a shorter self-contained presentation:
\begin{thm} 
\label{t:mercont}
Suppose that $ P := - \Delta_g - ( \frac{n-1} 2)^2 $ and that 
$ R ( \lambda ) := ( P - \lambda^2 )^{-1} \colon L^2 ( M ) \to L^2 ( M ) $, 
$ \Im \lambda > 0 $ is the resolvent. 
Then $ R ( \lambda ) $ continues meromorphically as an operator
\[  R ( \lambda ) : \CIc ( M ) \to \CI ( M ) .\]
Moreover, if $\lambda$ is a resonance of $P$, then there exists a non-trivial solution (resonant state) $v$ of $(P-\lambda^2)v=0$ which satisfies
\begin{equation}
\label{eq:outgoing}
  \tilde v = x^{\frac{n-1}{2}-i\lambda}v, \ \ v\in\CI(\bar M_{\mathrm{even}}),
\end{equation}
where $\bar M_{\mathrm{even}}=\bar M$ as topological spaces, but where smooth functions on $\bar M_{\mathrm{even}}$ (the `even compactification') are precisely those smooth functions on $\bar M$ which are smooth in $x^2$ near $\partial_1\bar M$.
\end{thm}

For $(M,g)=(\HH_1^n\setminus\cO,g_1)$, this is also discussed in \cite{bo:upper}. By rescaling, Theorem~\ref{t:mercont} applies to $ ( \HH_\kappa^n \setminus \mathcal O, g_\kappa ) $ as well, with $ P = P_\kappa $ given by \eqref{EqHypOp} and the resolvent denoted by $ R_\kappa ( \lambda ) $.  The multiplicity of a non-zero resonance $\lambda$ of $P_\kappa$ is then defined as
\begin{equation}
\label{eq:multk}
  m_\kappa(\lambda) = \dim\Bigl[\Bigl(\oint_\lambda R_\kappa(\zeta)\,d\zeta\Bigr)(L^2_\comp(\HH^n_\kappa\setminus\cO))\Bigr],
\end{equation}
where the contour is a small circle around $\lambda$, traversed counter-clockwise, which does not contain any other resonances.

\subsection{Resonance free strips for non-trapping obstacles for general 
hyperbolic ends}
\label{s:rfs}

The estimates on resonance width, $ | \Im \lambda | $, will be obtained
by studying local energy decay (see \cite{Mo2,fela,HZ} for
arguments which use the resonant states directly). The most conceptual way of
relating energy decay to resonances is via {\em resonance expansions of 
waves}; we will discuss this for general non-trapping obstacles on 
manifolds with asymptotically hyperbolic ends. In \S\ref{s:nontrap}, we shall prove that star-shaped obstacles in hyperbolic space are non-trapping.

For $ M$ given in \S \ref{s:mercont}, Melrose--Sj\"ostrand \cite{mels1,mels2} (see also \cite[Definition 24.3.7]{H3}) defined 
the broken geodesic flow. We make a general assumption
here that the geodesics do 
{\em not} have points of infinite tangency to $ \partial M $. 

A combination of \cite[Propositions 4.4, 4.6, Proof of Theorem 1.3]{burq} (see also 
\cite[\S3.3]{boule}) and \cite[Theorems 6.14, 6.15]{dizzy} immediately gives
\begin{thm}
\label{t:nontr}
Suppose that $ ( M , g ) $ is an asymptotically hyperbolic manifold with 
boundary. We assume 
that the geodesics do 
{\em not} have points of infinite tangency to $ \partial M $, and 
that the broken geodesic flow is {\em non-trapping}, that is, each geodesic leaves any compact set. Then for any $ \alpha > 0 $ and $ \chi \in \CIc ( M ) $
there exists $ C > 0 $ such that
\begin{equation}
\label{eq:resbound}  {\Im \lambda > - \alpha}  , \ \  { \Re \lambda \geq C }
\ \Longrightarrow \ 
 \| \chi R ( \lambda ) \chi \| \leq C |\lambda |^{-1} .
\end{equation}
In particular, there are only finitely many poles of $ R ( \lambda ) $ in 
any strip $ \Im \lambda > - \alpha $.
\end{thm}

\begin{rmk}
In the case of $ \mathbb H_\kappa^n \setminus \mathcal O $ we could get a stronger
result using Vainberg's method \cite{Vai} (see also \cite[\S 4.6]{dizzy}): namely 
a logarithmically large resonance free region. Since that improvement 
is not necessary for our arguments we opted for a more general version.
\end{rmk}

This immediately implies resonance expansions, see for example \cite[Proof of Theorem~5.10]{ZwBook}:

\begin{thm}
\label{t:expansion}
  Let $(M,g)$ be an asymptotically hyperbolic manifold satisfying the assumptions of Theorem~\ref{t:nontr}. Suppose that $u(t,x)$ is the solution of
  \begin{gather*}
      (D_t^2 - P_k)u(t,x) = 0\ \text{in}\ \R\times(\HH^n_\kappa\setminus\mathcal O),\ \ u(t,x) = 0\ \text{on}\ \R\times\partial\mathcal O, \\
      u(0,x) = u_0(x) \in H^1_\comp(\HH^n_\kappa\setminus\mathcal O), \ \ \partial_t u(0,x) = u_1(x) \in L^2_\comp(\HH^n_\kappa\setminus\mathcal O).
  \end{gather*}
  Denote by $\{\lambda_j\}\subset\C$ the set of resonances of $P_\kappa$. Then, for any $A>0$,
  \[
    u(t,x) = \sum_{\Im\lambda_j > -A} \sum_{\ell=0}^{m_\kappa(\lambda_j)-1} t^\ell e^{-i\lambda_j t}u_{j,\ell}(x) + E_A(t),
  \]
  where the sum is finite,
  \[
    \sum_{\ell=0}^{m_\kappa(\lambda_j)-1} t^\ell e^{-i\lambda_j t}u_{j,\ell}(x) = \Res_{\lambda=\lambda_j}\bigl(e^{-i\lambda t}(i R_\kappa(\lambda)u_1 + \lambda R_\kappa(\lambda)u_0)\bigr),
  \]
  $(P_\kappa-\lambda_j^2)^{k+1}u_{j,k}=0$, and for any $K>0$ such that $\supp u_j\subset B(0,K)$, there exist constants $C_{K,A}$ and $T_{K,A}$ such that
  \[
    \| E_A(t) \|_{H^1(B(0,K))} \leq C_{K,A}e^{-A t}(\|u_0\|_{H^1} + \|u_1\|_{L^2}),\ \ t \geq T_{K,A}.
  \]
\end{thm}

The remainder $E_A$ is only estimated in $H^1$ because \eqref{eq:resbound} only gives a \emph{strip} free of resonances, rather than a logarithmic region.

\subsection{Energy-stress tensor and the vector field method}
\label{s:est}
We briefly recall the general formalism for obtaining energy estimates,
referring to \cite[\S 2.6]{Tay1} and
\cite[\S4.1.1]{DafermosRodnianskiLectureNotes} for detailed presentations
(see also \cite[\S 1.1]{DyatlovQNMExtended} for a concise discussion relevant here). 
The general setting we use here makes the 
formulas more accessible and will be particularly useful in \S  \ref{s:moe}.

Let $ M$ be an $(n+1)$-dimensional smooth manifold and $ G $ a Lorentzian 
metric on $ M$, that is, a symmetric $ (0,2)$-tensor of signature $(n, 1) $.
The volume form, gradient, and divergence are defined as in Riemannian geometry,
and they give the d'Alembertian, $ \Box_G u = {\rm{div}}_G ( \nabla_G u ) $. 
The {\em stress--energy tensor} for a Klein--Gordon operator $
\Box_G - m^2 $ is a symmetric $ ( 0 , 2 ) $-tensor 
associated to $ u \in \CI ( M ) $:
\[ T_u ( X , Y ) := G( X ,
\nabla_G u ) G ( Y , \nabla_G u )   - {\textstyle \frac12}( G ( 
\nabla_G u , \nabla_G u )  + m^2 u^2 ) 
G ( X, Y ) 
, \]
$ X, Y \in \CI ( M ; T M ) $. To $ T_u $ and a vector field 
$ V $ we associate the {\em current}
$ J^V ( u ) \in \CI ( M ; T M ) $ by demanding that for all vector fields $X$,
\[  G ( J^V ( u ) , X ) = T_u ( V , X ) . \]
The key identity is
\[ ( \Box_G - m^2 ) u \,V u 
+ K^V ( \nabla_G u, \nabla_G u ) -
 {\textstyle \frac14} m^2 u^2 \tr_G ( \mathcal L_V G ) 
= {\rm{div}}_G\,J^V ( u ) 
, \]
where 
\[  K^V :=  {\textstyle \frac12} \mathcal L_V G -  {\textstyle \frac14}
G \tr_G( \mathcal L_V G ) . \]
The simplest version of this identity arises in the case that $ V $ is a Killing vector field,
i.e.\ satisfies $ \mathcal L_V G= 0 $. In this case
the divergence theorem gives
\begin{equation}
\label{eq:dive}
( \Box_G - m^2 ) u = 0 , \ \ \mathcal L_V G = 0 \
\Longrightarrow \ \int_{\partial \Omega } G ( J^V ( u ) , \mathbf n ) \, d S  = 0 , 
\end{equation}
where $ \Omega \subset M $ is an open subset with smooth boundary, $ \mathbf n 
$ is the unit outward normal vector and $ dS $ is the measure induced on 
$ \partial \Omega $ by $ d \!\vol_G $. The outward unit normal vector is defined 
by the conditions
\[ G ( \mathbf n , \mathbf n ) = 1 , \ \ G ( \mathbf n , X ) > 0 , \]
for any vector field $ X $ pointing out of $ \Omega $. It may blow up for null
hypersurfaces, but this is then compensated by the vanishing of $ dS $ --
see \cite[Appendix C]{DafermosRodnianskiLectureNotes}.

\section{Morawetz estimates in hyperbolic space}
\label{s:moe}

We will now apply the general formalism recalled in \S \ref{s:est} to 
scattering by obstacles. It follows the approach of Morawetz \cite{Moa,Moe}, see also \cite[Appendix~3]{LP}, to Euclidean scattering. However, our derivation 
of the generalization of her fundamental identity \cite[Lemma 3]{Moe} seems slightly different.

\subsection{Conjugated equation and a weighted energy inequality}
\label{s:conj}

The Lorentzian metric corresponding to the metric \eqref{EqHyp} is given by 
\begin{equation}
\label{eq:tildegkappa}
  \tilde g_\kappa  = -dt^2 + g_{\kappa},
\end{equation}
and we define 
\begin{equation}
\label{eq:Boxka}\Box_\kappa := \Box_{\tilde g_\kappa } + \kappa^2 \left( \frac{n-1}2 \right)^2 
.\end{equation}
We then consider a conjugated operator, described in the following lemma; it was already used implicitly in \S \ref{s:ball}. When no confusion is likely we write
\[   s = s_\kappa ,  \ \ g = g_\kappa, \ \ \tilde g = \tilde g_\kappa .\]
\begin{lemm}
\label{l:conj}
With the notation of \eqref{EqHyp} and \eqref{eq:Boxka}, we have 
\begin{equation}
\label{eq:crucial}
s^{2} s^{ \frac{n-1}2} \Box_\kappa s^{ -\frac{n-1}2} = \Box_{G}  - \frac{ ( n-1) ( n-3) }4  , 
\end{equation}
where 
\begin{equation}
\label{eq:gkap}
G = G_{\kappa } := s ( r)^{-2} ( - dt^2 + dr^2 ) + h .
\end{equation}
In addition, if 
$ w := s^{\frac{n-1}{2}}u $, then
\begin{equation}
\label{eq:energyk}
\begin{split} 
& \left( u_t^2 + | \nabla_{ g } u |^2  - \kappa^2 \left( \tfrac{n-1}2 \right)^2 u^2 \right) d\!\vol_{ g}  \\
&  \ \ \ \  = \left( w_t^2 + w_r^2 + s ^{-2} | \nabla_h w |^2_h + s^{-2} \tfrac{(n-1)(n-3)}4 w^2 \right) dr \, d \! \vol_h - \left( \tfrac{n-1}2 \tfrac{ s'} s  w^2 \right)_r dr \, d \! \vol_h .
\end{split}
\end{equation}
\end{lemm}
\begin{proof}
Since $ |\det \tilde g | =  s^{ 2 ( n-1) } $, we have 
$ \Box_{\tilde g } = - \partial_t^2 + 
s^{-{n-1}}  \partial_r s^{{n-1}} \partial_r + s^{-2} \Delta_h $. 
Hence we only need to compute
\begin{equation}
\label{eq:k2s}  \begin{split}  s^{\frac{n-1}2} ( s^{-{n-1}}  \partial_r s^{{n-1}} \partial_r )
s^{-\frac{n-1}2} 
& =( s^{-\frac{n-1}2} \partial_r s^{\frac{ n-1} 2 } ) ( 
s^{\frac{n-1}2} \partial_r s^{-\frac{ n-1} 2 } ) \\
& = 
\textstyle \left( \partial_r + {{\frac{n-1}2 } \frac{s'} s } \right) \left( 
\partial_r - {{\frac{n-1}2} \frac{s'} s } \right) \\
& = \textstyle \partial_r^2 - \frac { (n-1)^2} 4 \frac{ s'^2   } { s^2} - \frac{n-1}2 
\frac{ s'' s - s'^2 }{ s^2 } 
\\
& = \textstyle \partial_r^2 - \frac{ ( n-1)(n-3)} {4 s^2}  - \kappa^2 \left(\frac{n-1} 2\right)^2 .
\end{split} \end{equation}
Multiplying by $ s^2 $ gives \eqref{eq:crucial}. (We direct the reader to \eqref{EqHwConf} for a more conceptual point of view.) To establish \eqref{eq:energyk}, it again suffices to consider radial derivatives:
\[ \begin{split} 
( u_r^2 - \kappa^2 (\tfrac{n-1} 2 )^2 u^2 ) s^{n-1} 
& = s^{n-1} ( s^{ -\frac{n-1} 2 } w )_r^2 - \kappa^2 (\tfrac{n-1} 2 )^2 w^2 \\
& =
( w_r - \tfrac{n-1} 2 \tfrac{s'} s w)^2 -
\kappa^2 
(\tfrac{n-1} 2 )^2 w^2 
\\
& = w_r^2 - ( \tfrac{n-1}2 \tfrac {s'}s w^2 )_r + 
\left( \tfrac{n-1}2 \left(\tfrac{s'}s \right)_r + 
\left( \tfrac {n-1} 2 \right)^2 \left(\tfrac{s'}s \right)^2 - 
\kappa^2 \left( \tfrac {n-1} 2 \right)^2 \right) w^2
\\
&  = 
w_r^2 + \tfrac{(n-1)(n-3)}{4s^2}  w^2 -  ( \tfrac{n-1}2 \tfrac {s'}s w^2 )_r,
  \end{split} 
 \]
where we used the same computation as in \eqref{eq:k2s}. Since 
$ d \! \vol_{g_\kappa} = s^{n-1} d r d \! \vol_h $, \eqref{eq:energyk} follows. 
\end{proof}

\subsection{An energy identity}
\label{s:ane}

We now calculate the stress--energy tensor given by \eqref{eq:gkap} for the metric $ G $ and for $m^2=\frac{(n-1)(n-3)}{4}$ in terms of the decomposition of vectors into components of $\partial_t$, $\partial_r$, and vectors  tangent to the sphere:
\begin{align*}
T_u &= 
\begin{bmatrix} u_t^2  & u_r u_t & u_t \nabla_h u^T  \\
u_r u_t & u_r^2 & u_r \nabla_h u^T  \\
\nabla_h u u_t  & \nabla_h u u_r  & \nabla_h u \nabla_h u ^T \end{bmatrix} \\
 &\qquad\qquad - 
\frac{  s^2 ( - u_t^2 + u_r^2 ) + | \nabla_h u |_h^2 + m^2 |u|^2 }{2}
\begin{bmatrix} - s^{-2} & 0 & 0 \\ 0 & s^{-2} & 0 \\ 0 & 0 & 1 \end{bmatrix}.
\end{align*}
Let $ V := a ( t, r ) \partial_t + b ( t, r ) \partial_r $.
Then the orthogonal decomposition of $ J^V ( u ) $ with respect to 
$ dt^2 + dr^2 + h $ is given by 
\begin{equation}
\label{eq:JVu} \begin{split} J^V ( u ) & = \begin{bmatrix} -s^2 & 0 & 0 \\ 0 & s^2 & 0 \\
0 & 0 & 1 \end{bmatrix} T_u \begin{bmatrix} 
a \\ b \\ 0 \end{bmatrix} \\
& = \textstyle{\frac 12 } 
\begin{bmatrix} - s^2 ( a u_t^2 + 2 b u_r u_t + a u_r^2 ) -  
a |\nabla_h u |_h^2 - a m^2 |u|^2
\\
s^2 ( b u_r^2 + 2  a u_r u_t +  b u_t^2 ) - b | \nabla_h u|_h^2 - b m^2 |u|^2 \\
{2}a u_t \nabla_h u + {2}b u_r \nabla_h u 
\end{bmatrix}.
\end{split}
\end{equation}
We now calculate
\[ \mathcal L_V G = 2 s^{-2} \left( (- a_t + b s'/s ) dt^2 + ( b_t - a_r ) dt dr 
+ (  b_r  - b s'/s) dr^2 \right) , \]
and hence for $ V $ to be a Killing vector field, we need to find $ a = a ( t, r ) $ and $ b = b ( t , r ) $ such that
\begin{equation}
\label{eq:atbr}  a_t = b_r, \ \ a_r = b_t , \ \ b_r = b s'/s.
\end{equation}
An obvious choice is $a=1$, $b=0$, which in the context of the identity~\eqref{eq:dive} gives energy conservation. As will be clear later, a convenient choice for the purpose of proving generalized Morawetz (local energy decay) estimates is given by 
\begin{equation}
\label{eq:atbr1}
 a := 2 \kappa^{-2} ( \cosh \kappa r \cosh \kappa t - 1 ) , \ \
b := 2 \kappa^{-2} \sinh \kappa r \sinh \kappa t  .
\end{equation}

Suppose that $ w \in C ( \RR ; H^1_{{\comp}} ( \HH_k^n \setminus \mathcal O ) ) $.
We will use the following notation:
\begin{equation}
\label{eq:Fab}
\begin{split}
& F_{a,b} (w , T , R) := \\
& \ \ \ \int_{ B ( 0 , R)  \setminus 
\mathcal O}  \left(a w_t^2 + 2 b w_t w_r +  a w_r^2 + a s^{-2} | \nabla_h u |^2_h + a s^{-2} m^2 w^2 \right) dr \, d \! \vol_h |_{t=T}, 
\end{split}
\end{equation}
\[  F_{a,b} ( w, t ) := F_{a,b} ( w, t , \infty ), \ \  B ( 0 , \infty ) := \mathbb H^n_\kappa . \]

An application of \eqref{eq:dive} then gives
\begin{lemm}
\label{l:energy}
Suppose that $ a = a ( r, t ) $ and $ b = b ( r, t ) $ satisfy 
\eqref{eq:atbr} and that $ F_{a,b} $ is defined by \eqref{eq:Fab}. 
Suppose that $ w \in C ( \RR ; H_{\comp}^1 ( \HH^n_\kappa \setminus 
\mathcal O ) ) $ satisfies
\[
  ( \Box_G -m^2 ) w ( t, x ) = 0 ,  \ \ x \in \mathbb H_\kappa^n \setminus \mathcal O , \ \ w ( t, x ) = 0 , \ \ x \in \partial \mathcal O .
\]
Then
\begin{equation}
\label{eq:energyl2} 
F_{a,b} ( w , T ) = F_{a,b} ( w , 0 ) -  \int_0^T \int_{ \partial \mathcal O} 
b ( t, r ) g ( \nu, \nabla_g r ) g ( \nu, \nabla_g w )^2  s ( r )^{-n+1} d \sigma_g \, dt , 
\end{equation}
where $ g = g_\kappa $ is given by \eqref{EqHyp}, $ \nu $ is the outward unit
(with respect to $ g $) normal vector and $ d \sigma_g $ is the 
measure induced by $ g $ on $ \partial \mathcal O $.
\end{lemm}
\begin{proof}
Since $ \mathcal O $ is star-shaped (with respect to the origin, as can be assumed without loss of generality)
we can write 
\[ \partial \mathcal O = \{ x = r \omega:  r = f ( \omega ) , \ \omega \in \SP^{n-1} \} \]
for some $ f : \SP^{n-1} \to ( 0 , \infty ) $. To obtain \eqref{eq:energyl2}, we
apply \eqref{eq:dive} to $ \Omega = [ 0 , T ] \times ( \HH^n_\kappa \setminus 
\mathcal O )$, 
$$ \partial \Omega = \Gamma_1 \cup \Gamma_2 , \ \ \ 
\Gamma_1 : = [ 0 , T ] \times \partial \mathcal O , \ \  \Gamma_2 := \{ 0, T \} \times 
( \HH^n_\kappa \setminus \mathcal O ) 
. $$
On $\Gamma_1  $ (in the notation of \eqref{eq:JVu}),
\[  \mathbf n =  -\nabla_G ( r - f ) / | \nabla_G ( r- f )|_G =  [ 0 , -s^2 , \nabla_h f ] ( s^2 +  |\nabla_h f |^2_h )^{-\frac12} , \ \  s = s (f ( \omega ) ) , \]
\[ d S = s^{-2} (s^2 + | \nabla_h f |^2_h )^{\frac12} d \vol_h dt ,\]
\[ \begin{split} 
 g ( J^V ( w ) , \mathbf n ) & = - {\textstyle \frac12 } b ( s^2 w_r^2 - | \nabla_h w |_h^2 -  {2} w_r \langle \nabla_h w , \nabla_h f \rangle_h ) ( s^2+  |\nabla_h f |^2_h )^{-\frac12}  \\
& = - {\textstyle \frac12 } {{b w_r^2}} ( s^2 +  |\nabla_h f |^2_h )^{{\frac12} }
, \end{split} \]
where we used the boundary condition $ w (t, f ( \omega ), \omega ) \equiv  0 $
(thus $ w_t ( t, f ( \omega ) , \omega ) = 0 $ and 
$ (\nabla_h w ) ( t , f ( \omega ) , \omega ) = - w_r \nabla_h f ( \omega ) $.)
On $ \Gamma_2 $, 
\[ \mathbf n = \pm [ s , 0 , 0 ] ,  \ \ 
d S = s^{-1} dr d \! \vol_h , \]
\[  g ( J^V ( w ) , \mathbf w ) = \mp  
{\textstyle{\frac12} } s \left( a w_t^2 + b w_r w_t + a w_r^2 + a 
s^{-2 } | \nabla_h w |^2_h +  a m^2 |w|^2 \right)  .\]
Hence, \eqref{eq:dive} gives
\[ F_{ a, b } ( w, T ) = F_{a,b} ( w, 0 ) -   \int_0^T \int_{\SP^{n-1} } 
b (t , f ( \omega ))  ( 1 + s^{-2} | \nabla_h f |^2_h ) w_r ( f ( \omega) , \omega )^2\, d \! \vol_h d t . \]
The more invariant form given in \eqref{eq:energyl2} 
follows from explicit expressions: 
\[ \nu = ( 1 + s^{-2} | \nabla_h f |^2_h )^{-\frac12} ( 1 , - 
s^{-2} \nabla_h f ) , \ \ 
d \sigma_g = s^{n-1} ( 1 + s^{-2} | \nabla_h f |^2_h)^{\frac12} 
d \vol_h, \]
and $ g ( \nu ,  \nabla_g w ) = w_r ( 1 + s^{-2} |\nabla_h f |^2 )^{\frac12}$. 
\end{proof}

We remark that \eqref{eq:energyl2} is valid for any obstacle $ \mathcal O$; the assumption that $\partial\mathcal O$ be star-shaped implies however that the second term on the right hand side is negative.

\subsection{Proof of Theorem 1 for \texorpdfstring{$ n \geq 3 $}{n at least 3}}
\label{s:pt1}

Let $ a, b $ be given by \eqref{eq:atbr1}.
Lemmas \ref{l:conj} and \ref{l:energy} give the following energy inequality: 
Suppose that, in the notation \eqref{eq:Boxka}, 
\begin{gather*}   \Box_\kappa u = 0  \ \text{ in } \  [ 0, T ] \times \mathbb H_\kappa^n 
\setminus \mathcal O , \ \ u  = 0  \ \text{ on } \  [ 0, T ] \times \partial \mathcal O, \\
( u ( 0 , x ) , u_t ( 0 , x ) ) \in (H^1 \times L^2) ( \mathbb H_\kappa^n \setminus \mathcal O ) ,
\end{gather*} 
and define
\begin{equation}
\label{eq:Eab}
  E_{a,b}(u,t,r) := F_{a,b}(s^{\frac{n-1}{2}}u, t, r), \ \ E_{a,b}(u,t) := E_{a,b}(u,t,\infty),
\end{equation}
then we have
\[ E_{a,b} ( u, t ) \leq E_{a,b} ( u, 0 ) ,\ \ t\geq 0. \]
We now assume that
\[   \supp \partial_t^k u ( 0 , \bullet ) \subset B (0 , R ), \ \ k=0,1. \]
Since 
for any $a,b\geq 0$,
\[   ( a - b ) ( x^2 + y^2 ) \leq a x^2 + 2 b xy + a y^2 \leq ( a + b ) ( x^2 + y^2 ) \]
and since 
for  $ n \geq 3 $,  $ m^2 = ( n-1)(n-3)/4 \geq 0 $, we then have 
\begin{gather*} 
  E_{ c, 0 } ( u , t, R ) \leq E_{a,b}(u,t,R), \ \  t \geq 0 , \\ 
   c ( t, r ) := a (t,r) - b ( t, r)  = 2 \kappa^{-2}  (\cosh \kappa ( t - r ) -1) 
\end{gather*}
Hence for $ t \geq R $, and using that $a=c$ and $b=0$ at $t=0$,
\[
\begin{split}  c ( t , R ) E_{1,0} ( u ,t , R ) & \leq E_{c, 0} ( u, t ,R ) 
\leq E_{a,b}(u,t,R) \\
& \leq E_{a,b} ( u  , 0 ) = E_{c, 0} ( u, 0 ) \leq c ( 0, R ) E_{1,0} ( u, 0 ) .
\end{split}
\]
We obtain for $ t >  2 R $:
\begin{equation}
\label{eq:EutR}   E_{1,0} ( u , t , R ) \leq  \frac{ \cosh \kappa R - 1 } { \cosh \kappa ( t - R ) - 1 } 
E_{1,0} ( u , 0 ) \leq C e^{-\kappa t}E_{1,0}(u,0).
\end{equation}
The results of \S \ref{s:rfs} then give Theorem \ref{t:1}. This crucially uses that $E_{1,0}(u,t,R)$ is coercive (unlike the integral of the natural energy density for $u$, given by the left hand side of the identity~\eqref{eq:energyk}, over $B(0,R)\setminus\mathcal O$).

We note that the second part of Lemma \ref{l:conj} shows that for the quantity
\[
  \mathcal E( u, t, R ) := \int_{B(0,R)\setminus\mathcal O}\left(u_t^2+|\nabla_g u|^2-\kappa^2\left(\tfrac{n-1}{2}\right)^2u^2\right)\,d\mathrm{vol}_g\,|_{t=T},
\]
$\mathcal E(u,T):=\mathcal E(u,T,\infty)$, we have
\[   \mathcal E ( u , t ) = E_{1,0} ( u , t ) , \ \  \mathcal E ( u , t , R ) \leq E_{1,0} ( u, t, R ) . \]
However, $\mathcal E(u,t,R)$ is only coercive when $\kappa=0$, and in this case, the argument leading to~\eqref{eq:EutR} gives the estimate of Morawetz:
\begin{equation}
\label{eq:EutRM}  \mathcal E ( u , t , R ) \leq  \frac{ R^2 }{ ( t - R )^2 } \mathcal E ( u , 0 ) , \end{equation}
when the initial data have support in $ B ( 0 , R) \setminus \mathcal O $.

\section{Improved estimates in odd dimensions}
\label{s:imp}

We now revisit the argument of Morawetz for obtaining exponential decay 
in odd dimensions. 
We use the notation of \eqref{EqHyp} and \eqref{eq:Boxka} and
we denote by $ \nabla_\kappa $, $ | \bullet |_{\kappa } $ the 
gradient and norm with respect to the Riemannian metric on $ \HH^n_\kappa $. We recall that the obstacle $\mathcal O$ is star-shaped with respect to the origin $0\in\HH^n_\kappa$, and we assume that $\mathcal O$ is contained in the ball $B(0,\rho)$.

The key fact is the strong Huyghens principle illustrated in Fig.~\ref{f:mor_cone}: suppose that
\[
  \Box_\kappa u = F , \ F \in \mathcal D' ( \RR \times \HH^n_\kappa ) , \ \ 
u ( t , x )|_{ t < 0 } = 0;
\]
if
\[ \Gamma_{ ( t , x )}^- := \{ ( t', x') \colon t'\leq t,\ d_\kappa ( x,x' ) = t - t' \} \]
then
\begin{equation}
\label{eq:sHp}
\Gamma_{ ( t , x)}^{-} \cap \supp F = \emptyset , \ 
\ \Longrightarrow \ u (t, x ) = 0  .
\end{equation}

\begin{figure}
\includegraphics{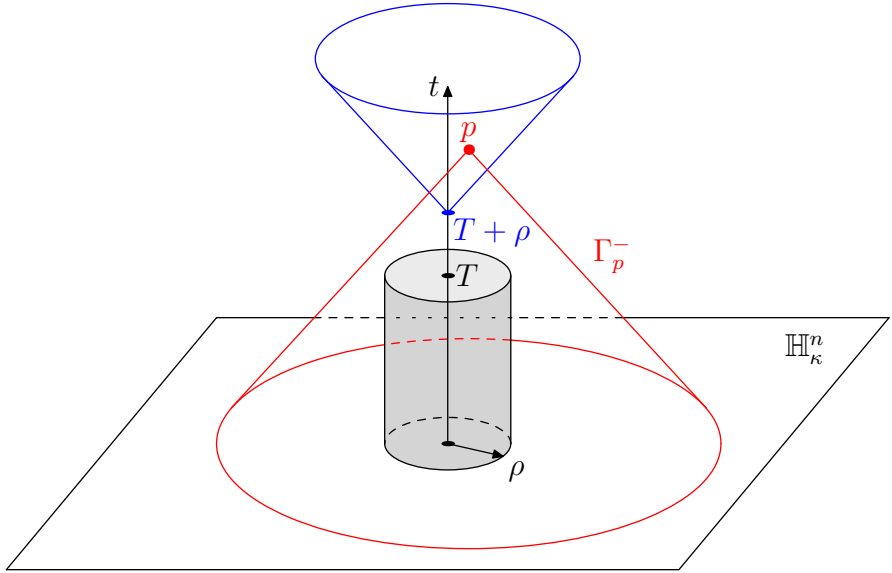}
\caption{If $ \Box_{\kappa} u = F $ in $ \RR_+ \times \HH_\kappa^n $, $ n\geq 3 $, odd,
and $ F$ is supported in 
$ B_{T, \rho }:= [ 0 , T ] \times B ( 0 , \rho ) $ (shaded region), and $u|_{t=0}$ and $\pa_t u|_{t=0}$ are supported in $B(0,T+\rho)$, then $ u ( t , x) \equiv 0 $ for $ |x| \leq t - (T + \rho) $.}
\label{f:mor_cone}
\end{figure}

We will make use of the (local) energy
\[
  E(u,t,r):=E_{1,0}(u,t,r), \ \ E(u,t):=E(u,t,\infty),
\]
defined using \eqref{eq:Eab}. For $ u ( t , \bullet ) $ defined on $ \mathbb H^n_\kappa $, we can also integrate over $ B ( 0 , r ) $ and we denote the corresponding energies by $ E_0 ( u, t , r ) $, $ E_0 ( u, t ) $.

We will now consider 
\begin{equation}
\label{eq:initu}  \begin{gathered}  \Box_\kappa u = 0 \ \text{on} \ 
 [ 0 , \infty ) \times \HH^n_\kappa \setminus \mathcal O , \ \ 
 u|_{ ( 0 , \infty ) \times \partial \mathcal O } = 0 \\
  \supp \partial_t^k u|_{t=0} \subset 
B ( 0 , R ) , \ \ 
 \partial_t^k u|_{t=0}  \in H^{1-k} ( \HH^n_\kappa \setminus \mathcal O ) , \ \ k = 0,1 . 
\end{gathered} \end{equation}
For solutions $ u $, the energy $ E ( u , t ) $ does not depend on time.

\begin{lemm}
\label{l:mor2}
Let $ u $ be the solution to the initial value problem 
\eqref{eq:initu} with $ R = 3\rho $. Then for any $ T \geq 2\rho $, we have a decomposition
\begin{equation}
\label{eq:uTrT}
\begin{gathered}
u ( t , x ) = u_T ( t , x ) + r_T ( t, x ) \ \text{ for }  \ t \geq T, \\ 
u_T ( t , x) = 0 \ \text{ for }  \  d ( x , 0 ) \geq t - (T - \rho) ,\\
r_T ( t, x ) = 0 \ \text{ for }  \  d ( x, 0 ) \leq t - (T + \rho) , \\
E ( u_T , T + 2\rho ) \leq  4 E ( u , T , 3 \rho ).
\end{gathered}
\end{equation}
\end{lemm}
\begin{proof}
If  
$
 u|_{t = T } =: f_T \in H^1_0 ( \HH^n_\kappa \setminus 
\mathcal O) $, and $ \partial_t u|_{t = T } =: g_T \in L^2  ( \HH^n_\kappa \setminus 
\mathcal O)$, 
we define
$ \tilde f_T \in H^1 ( \HH^n_\kappa ) $ and $ g_T \in L^2 ( \HH^n_\kappa ) $ by 
extending $ f_T $ and $ g_T $ by $ 0 $ to $ \mathcal O $. 
We then solve the free equation
\[
  \Box_\kappa r_T = 0 \ \text{on} \ [ T, \infty ) \times \HH^n_\kappa, \ \
r_T |_{t=T} = \tilde f_T , \  \ \partial_t r_T |_{ t = T } = \tilde g_T .
\]
To prove the support condition on $r_T$, define $\tilde u\in L^2(\RR\times\HH^n_\kappa)$ to be equal to $u$ in $(0,\infty)\times\HH^n_\kappa\setminus\mathcal O$, and equal to $0$ otherwise; let then $F=1_{t<T}\Box_\kappa\tilde u$, which has $\supp F\subset[0,T]\times\mathcal O$. Then the forward solution of $\Box_\kappa\tilde r_T=F$ is equal to $\tilde u$ in $t<T$, hence has the same Cauchy data as $r_T$ at $t=T$, and we conclude that $r_T=\tilde r_T$ in $t\geq T$; it remains to apply \eqref{eq:sHp}. (See Fig.~\ref{f:mor_cone}.)

We then see that $ u_T := u - r_T $ solves the mixed problem
\[ \begin{gathered} \Box_\kappa u_T = 0 \ \text{on} \ [ T, \infty ) \times \HH^n_\kappa \setminus 
\mathcal O, \ \
u_T |_{t=T} = 0 , \  \partial_t u_T |_{ t = T } = 0,  \\ 
u_T |_{ [T, \infty ) \times \partial \mathcal O } = - r_T |_{ [ T , \infty ) \times \partial \mathcal O} . \end{gathered} \]
It remains to estimate the energy of $ u_T $. Note that the support of $u_T$ and $\pa_t u_T$ at $t=T+2\rho$ is contained in $B(0,3\rho)$. Thus, using the Killing vector field $\partial_t$ (for the metric $G$ and the function $w=s^{\frac{n-1}{2}}u$) to obtain energy estimates, we have
\begin{equation}
\label{eq:naive}
\begin{split}
  E(u_T,T+2\rho) &\leq 2 E(u, T+2\rho, 3\rho) + 2 E(r_T, T+2\rho, 3\rho) \\
    &\leq 2 E(u, T, 5\rho) + 2 E_0(r_T, T, 5\rho) = 4 E(u, T, 5\rho). 
\end{split}
\end{equation}

The improved estimate in \eqref{eq:uTrT} is obtained as follows. The boundary data $ -r_T |_{ [ T , \infty ) } $ of $u_T$ depend only on the values of $ r_T $ in the backwards solid cone slice
\[ \mathcal C_{ T, \rho} := \{ ( t , x ) : d ( x, 0 ) \leq T + 3 \rho - t ,  \ T \leq t \leq T + 2\rho \} \]
and hence on $u(T,\bullet)$ and $u_t(T,\bullet)$ in $ B ( 0 , 3 \rho ) \setminus
\mathcal O $ -- see Fig.~\ref{f:cone2}. In \eqref{eq:naive}, we estimated the energy of $u_T$ by writing it as the difference of two solutions, $u$ and $r_T$, each of which satisfied simple energy estimates that did not involve data on timelike boundaries. In order avoid contributions from outside $B(0,3\rho)$, we place a timelike boundary at $[T,T+2\rho]\times\partial B(0,3\rho)$, which does not affect waves inside the cone $\mathcal C_{T,\rho}$. Thus, consider the boundary value problem
\begin{gather*}
  \Box_\kappa r'_T = 0\ \text{on}\ [T,T+2\rho]\times B(0,3\rho),  \ \ r'_T |_{t=T} = r_T, \ \partial_t r'_T|_{t=T}=\pa_t r_T, \\
  r'_T|_{[T,T+2\rho]\times\partial B(0,3\rho)} = r_T|_{\{T\}\times\partial B(0,3\rho)}.
\end{gather*}
Note that the data on the artificial boundary are independent of $t$.\footnote{Since we are working on the level of $H^1$ only, there are no additional compatibility conditions on $\partial_t r'_T$ at $\{T\}\times\partial B(0,3\rho)$. One can also see this directly by taking $\partial_t r'_T|_{t=T}=1_{r<3\rho-\delta}\partial_t r_T$ for $\delta>0$ small, and letting $\delta\to 0$.} The above domain of dependence argument implies $r'_T|_{[T,T+2\rho]\times\partial\mathcal O}=r_T|_{[T,T+2\rho]\times\partial\mathcal O}$. Moreover, since $\pa_t r'_T\equiv 0$ on the artificial boundary $[T,T+2\rho]\times\partial B(0,3\rho)$, we have the energy identity
\[
  E(r'_T, T+2\rho) = E(r'_T, T) = E(r_T, T, 3\rho) = E(u, T, 3\rho),
\]
where we measure the energy of $r'_T$ in $B(0,3\rho)$. On the other hand, the function $u'_T$, defined by
\begin{gather*}
  \Box_\kappa u'_T = 0 \ \text{on} \ [ T, T+2\rho ] \times B(0,3\rho) \setminus \mathcal O, \ \ u'_T |_{t=T} = 0 , \  \partial_t u'_T |_{ t = T } = 0,  \\ 
 u'_T |_{ [T, T+2\rho] \times \partial \mathcal O } = - r'_T |_{ [ T , T+2\rho] \times \partial \mathcal O}, \ \ u'_T|_{[T,T+2\rho]\times\partial B(0,3\rho)} = 0,
\end{gather*}
is equal to $u_T$ in $[T,T+2\rho]\times B(0,3\rho)\setminus\mathcal O$. To estimate the energy of $u_T$, hence $u'_T$, at $t=T+2\rho$ in $B(0,3\rho)$, we note that $u':=u'_T+r'_T$ solves the wave equation
\begin{gather*}
  \Box_\kappa u' = 0\ \text{on}\ [T,T+2\rho]\times B(0,3\rho)\setminus\mathcal O, \ \ u'|_{t=T}=r_T,\ \partial_t u'|_{t=T}=\partial_t r_T, \\
  u'|_{[T,T+2\rho]\times\partial\mathcal O}=0, \ u'|_{[T,T+2\rho]\times\partial B(0,3\rho)}=r_T|_{\{T\}\times\partial B(0,3\rho)}
\end{gather*}
Thus, $u'$ satisfies the energy identity
\[
  E(u',T+2\rho) = E(u',T) = E(u,T,3\rho);
\]
and therefore we have
\[
  E(u_T,T+2\rho,3\rho) \leq 2 E(u',T+2\rho) + 2 E(r'_T,T+2\rho) = 4 E(u,T,3\rho),
\]
as claimed in \eqref{eq:uTrT}.
\end{proof}

\begin{figure}
\includegraphics[width=5in]{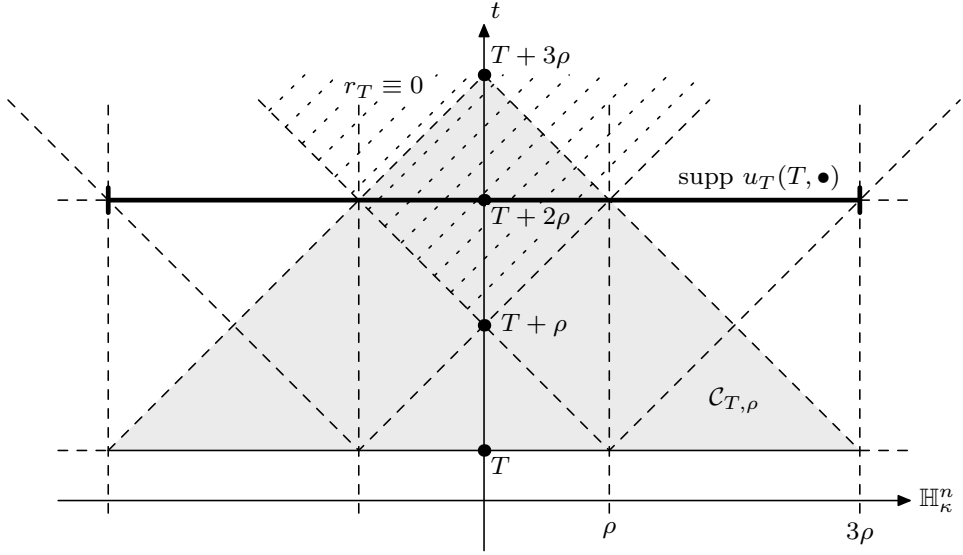}
\caption{Domain of dependence relations for $ u_T $.}
\label{f:cone2}
\end{figure}

\begin{lemm}
\label{l:mor3}
Suppose that $ n \geq 3 $ is {\em odd} and that solutions $ u $ of \eqref{eq:initu} with $ R = 3 \rho $
satisfy
\begin{equation}
\label{eq:weakdec}   E ( u , t  , 3 \rho ) \leq p ( t ) E ( u , 0 ) ,  \ \ t > 0 , \end{equation}
for some decreasing function $ t \mapsto p ( t ) $.
Then solutions to \eqref{eq:initu}  satisfy
\begin{equation}
\label{eq:strongdec} 
E ( u , t , 3 \rho ) \leq  C e^{ - 2\alpha t } E ( u , 0 ) , \ \ t > 0 , 
\end{equation}
where
\begin{equation}
\label{eq:defalpha}   \alpha = - \min_{ \tau \geq 4 \rho }  \frac{ \log (4 p ( \tau - 2\rho )) }{ 2 \tau }. 
\end{equation}
\end{lemm}
\begin{proof}
Suppose $ u_T $ is given by Lemma \ref{l:mor2}. We can then apply
\eqref{eq:weakdec} to $ u_T $ (with the time origin shifted by $ T+2\rho $) to obtain
\begin{equation} 
\label{eq:EuTt}  E ( u_T , T + 2\rho + t , 3 \rho ) 
\leq p ( t ) E ( u_T , T + 2\rho ) \leq 4 p ( t ) E ( u , T , 3 \rho),
\end{equation}
provided $T\geq 2\rho$; for the first inequality we use that the support of the Cauchy data of $u_T$ at $t=T+2\rho$ is contained in $B(0,3\rho)$. From the support properties of $ r_T $ we see that $ u ( t, x ) = u_T ( t, x ) $ if $ d ( 0 , x ) \leq t - (T + \rho) $ and hence $ u_T ( T+2\rho+t , x) = u ( T+2\rho+t , x)  $ for $ d ( x, 0 ) \leq 3 \rho $ and $ t \geq 2\rho $. This and \eqref{eq:EuTt} (with $ t = \tau - 2\rho $, and $\tau\geq 4\rho$) imply that
\[
  E ( u, T + \tau  , 3 \rho ) \leq 4 p ( \tau - 2\rho ) E ( u, T , 3 \rho) , \ \ \tau \geq 4\rho .
\]
Starting with $T=\tau-2\rho$, $\tau\geq 4\rho$, with $E(u,T,3\rho)\leq p(\tau-2\rho)E(u,0)$, and iterating this estimate we see that
\[  E ( u , n\tau , 3 \rho ) \leq 4^{n-1}  p ( \tau - 2\rho )^n E ( u , 0 ) , \ \ 
\tau \geq 4 \rho ,\]
from which the conclusion \eqref{eq:strongdec} is immediate.
\end{proof}

The function $ t \mapsto p ( t ) $ appearing in \eqref{eq:weakdec} is given 
by \eqref{eq:EutR} and \eqref{eq:EutRM}:
\[
 p ( t ) = \left\{ \begin{array}{ll}  (\cosh (3 \kappa \rho) -1)/( \cosh(\kappa(t - 3 \rho)) - 1 ), & \kappa > 0 , \\ 
\ \ \ \ \ \ \ \ \ 9 \rho^2 / ( t - 3 \rho )^2 , & \kappa = 0 , \end{array} \right.
\]
for $t>3\rho$.

For $ \kappa = 0 $, $ \alpha ( \rho ) = \rho^{-1} \alpha ( 1 ) $ and we obtain (taking into account that for $\rho=1$ our expression for $p(\tau-2)$ is only valid for $\tau-2>3$)
\begin{equation}
\label{eq:eucl}
\alpha = - \frac 1 \rho \min_{ t > 5 } 
\frac{ \log 36  - 2 \log ( t - 5 ) }{2 t}  = \frac{\mu}{ \rho} ,
\end{equation}
which gives
\[
  \mu \simeq 0.0482,
\]
more than twenty times worse than the bound \eqref{eq:Ralston} obtained using
complex analysis methods applied to the scattering matrix \cite{RalstonDecay}.

For $\kappa>0$, we put $ \tau = (t + 5)\rho $ in \eqref{eq:defalpha}. This gives 
\begin{equation}
\label{eq:unialph}
  \alpha = \frac{A(\kappa\rho)}{\rho},
\end{equation}
where
\[
  A(\tilde\rho) := \max_{t>0} a(\tilde\rho,t), \ \ a(\tilde\rho,t) := \frac{1}{2(t+5)}\log\Bigl(\frac{\cosh(\tilde\rho t)-1}{4(\cosh(3\tilde\rho)-1)}\Bigr).
\]
Letting $t\to\infty$, one finds $A(\tilde\rho)\geq\tfrac12 \tilde\rho$, hence \eqref{eq:unialph} recovers $\alpha\geq\tfrac12 \kappa$. We get an improvement over this unconditional rate $\kappa$ when $A(\tilde\rho)>\tfrac12 \tilde\rho$, which happens when there exists $t>0$ such that
\[
  \frac{1}{2}(1-e^{-\tilde\rho t})^2 > 4 e^{5\tilde\rho}\bigl(\cosh(3\tilde\rho)-1\bigr);
\]
this has a solution if the right hand side is $\ <1/2$, which happens for $\tilde\rho<0.1221$. One can show that $A(\tilde\rho)$ is monotonically increasing, and $A(\tilde\rho)\to\mu$ as $\tilde\rho\to 0$. See Fig.~\ref{FigA}. Thus, we obtain the unconditional gap $\alpha>\mu/\rho$.

\begin{figure}[!ht]
  \centering
  \includegraphics[width=0.4\textwidth]{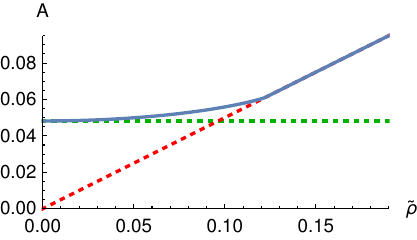}
  \caption{Estimates for the resonance width $\alpha=A/\rho$. \textit{Blue:} graph of $A(\tilde\rho)$. \textit{Green:} $A\equiv\mu$, corresponding to the Euclidean estimate~\eqref{eq:eucl}. \textit{Red:} $A=\tfrac12 \tilde\rho$, corresponding to the unconditional bound~\eqref{eq:t1}.}
  \label{FigA}
\end{figure}

\section{Hyperbolic space and general relativity} 
\label{s:geome}

The connection between hyperbolic space and de Sitter space of general relativity was emphasized by Vasy \cite{vasy2,vasy1} in his approach to the meromorphic continuation of the resolvent on asymptotically hyperbolic spaces (see \S \ref{s:mercont}). The key aspect which will be used in \S \ref{s:rela} is 
the characterization of resonant states as solutions to a conjugated equation
which extend smoothly across the boundary at infinity. We begin by reviewing
explicit connections between various models.

\subsection{Models of de~Sitter space}

Let $\kappa>0$. De~Sitter space in $(1+n)$ dimensions is the manifold $\dS_\kappa^{n+1}=\R_{t_0}\times\Sph^n$ with the metric
\[
  g_{\dS_\kappa^{n+1}} = -dt_0^2 + \bigl(\kappa^{-1}\cosh(\kappa t_0)\bigr)^2 H,
\]
where $ H $ is the usual metric on $ \SP^n $. This is an Einstein metric, $\Ric(g_{\dS_\kappa^{n+1}})=n\kappa^2 g_{\dS_\kappa^{n+1}}$, hence the scalar curvature is $R_{g_{\dS_\kappa^{n+1}}}=n(n+1)\kappa^2$.

First, we introduce the conceptually useful \emph{Einstein universe} $E^{n+1}=\R_s\times\Sph^n$, equipped with the metric $g_{E^{n+1}}=-ds^2+H$. If we take $s=2\arctan(e^{\kappa t_0})\in(0,\pi)$, so $\kappa^{-1}\cosh(\kappa t_0)ds=dt_0$, then
\[
  g_{\dS_\kappa^{n+1}} = \bigl(\kappa^{-1}\cosh(\kappa t_0)\bigr)^2 g_{E^{n+1}} = (\kappa\sin s)^{-2} g_{E^{n+1}}.
\]

\begin{figure}[!ht]
  \centering
  \includegraphics{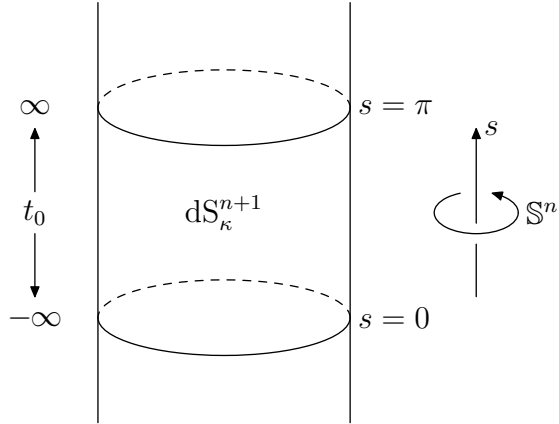}
  \caption{The Einstein universe $E^{n+1}=\R_s\times\Sph^n$, with $\dS_\kappa^{n+1}$ conformally diffeomorphic to the finite cylinder $(0,\pi)\times\Sph^n$.}
  \label{FigdSE}
\end{figure}

The coordinate change $t_M=\kappa^{-1}\sinh(\kappa t_0)$, so $dt_0=(1+\kappa^2 t_M^2)^{-1/2}\,dt_M$, expresses the de~Sitter metric as
\[
  g_{\dS_\kappa^{n+1}} = -\frac{dt_M^2}{1+\kappa^2 t_M^2} + (\kappa^{-2}+t_M^2)H,
\]
which is equal to the metric on the two-sheeted hyperboloid $\{ t_M^2 - |x_M|^2 = -\kappa^{-2} \} \subset (\R^{1+(n+1)}_{t_M,x_M}, -dt_M^2+dx_M^2)$ within Minkowski space, as can be seen by parametrizing $\dS_\kappa^{n+1}$ using the map $\R\times\Sph^n\ni(t_M,\omega)\mapsto(t_M,(\kappa^{-2}+t_M^2)^{1/2}\omega) \in \R^{1+(n+1)}$.

Next, we introduce the \emph{upper half space model}: define the map
\begin{equation}
\label{EqdSUHCoordChange}
  U^{n+1} := (0,\infty)\times\R^n \ni (\tau,x) \mapsto \Bigl(\frac{1-\kappa^2(\tau^2-|x|^2)}{2\kappa^2\tau}, \frac{1+\kappa^2(\tau^2-|x|^2)}{2\kappa^2\tau}, \frac{x}{\kappa\tau} \Bigr) \in \R^{2+n},
\end{equation}
where we write points in $\R^{2+n}$ as $(t_M,x_{M1},x'_M)$, i.e.\ splitting $x_M=(x_M^1,x'_M)$. This map is a diffeomorphism from the upper half space $U^{n+1}$ to the subset $\dS_{\kappa,*}^{n+1}=\{t_M+x_M^1>0\}\cap\dS_\kappa^{n+1}$ of de~Sitter space, and the de~Sitter metric takes the simple form
\begin{equation}
\label{EqdSUHMetric}
  g_{\dS_\kappa^{n+1}} = \frac{-d\tau^2+dx^2}{\kappa^2\tau^2}.
\end{equation}
The map~\eqref{EqdSUHCoordChange} is the inverse of the map
\begin{equation}
\label{EqdSUHMotivationMap}
  \dS_\kappa^{n+1} \ni (t_M,x_M^1,x'_M) \mapsto \Bigl(\frac{\kappa^{-2}}{x_M^1+t_M}, \frac{x'_M\kappa^{-1}}{x_M^1+t_M} \Bigr),
\end{equation}
defined for $x_M^1+t_M>0$, from which one deduces that the set $\dS_{\kappa,*}^{n+1}\subset\dS_\kappa^{n+1}$ in which the coordinates $(\tau,x)$ are valid is the causal future of the set $x_M^1+t_M=0$ within $\dS_\kappa^{n+1}$. As we will see below, this is equal to the causal future, within the Einstein universe, of the point $i^-$ at the past conformal boundary of $\dS_\kappa^{n+1}$, given by $(1,0)\in\Sph^n\hra\R^{n+1}$ at $s=0$. (We remark that the map \eqref{EqdSUHMotivationMap}, when instead restricted to $\HH^{n+1}_\kappa=\{t_M^2-|x_M|^2=\kappa^{-2},\ t_M>0\}$, takes one component of the two-sheeted hyperboloid in Minkowski space to the upper half space model of hyperbolic space.)

\begin{figure}[!ht]
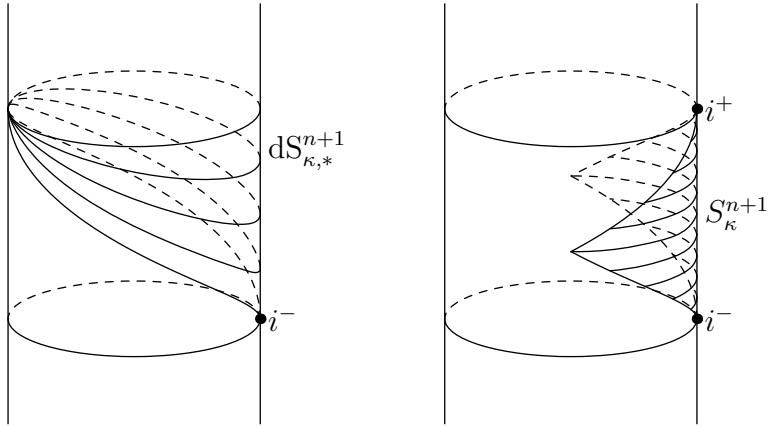

  \centering
  \includegraphics{dSUH} \hspace{0.4in} \includegraphics{dSStaticE}
\caption{\textit{Left:} the region $\dS_{\kappa,*}^{n+1}$ where the coordinates of the upper half space model are valid, within the Einstein universe, bounded in the past by the future light cone emanating from $i^-$. Also shown are level sets of the function $\tau$. \textit{Right:} the static patch $S_\kappa^{n+1}$ of de Sitter space (see \eqref{eq:static}) as a subset of the embedding of $\dS_\kappa^{n+1}$ into the Einstein universe $E^{n+1}$, together with level sets of $s$ (or $t_M$) within $S_\kappa^{n+1}$.}
\label{FigdSUH}
\end{figure}

For our purposes, the connection of hyperbolic space and de~Sitter space, exhibited in equation~\eqref{eq:HindS} below, takes place in the \emph{static model of de~Sitter space}, which we proceed to define. Fix the point $i^+_0=(1,0)\in\Sph^n\hra\R^{n+1}$, thought of as lying in the conformal boundary of $\dS_\kappa^{n+1}$ at future infinity. The static patch of de~Sitter space is the open submanifold
\begin{equation}
\label{eq:static}
  S_\kappa^{n+1} = \{ (t_M,x_M^1,x_M') \in \dS_\kappa^{n+1} \colon x_M^1\geq 0,\ |x_M'|<\kappa^{-1} \},
\end{equation}
see Fig.~\ref{FigdSUH} and Fig.~\ref{FigdSStatic}. $S_\kappa^{n+1}$ is the static patch of an observer who limits to the point $i^+=(\pi,i^+_0)\in E^{n+1}$ at future infinity and to the point $i^-=(0,i^-_0)$ at past infinity.

\begin{figure}[!ht]
  \centering
  \includegraphics{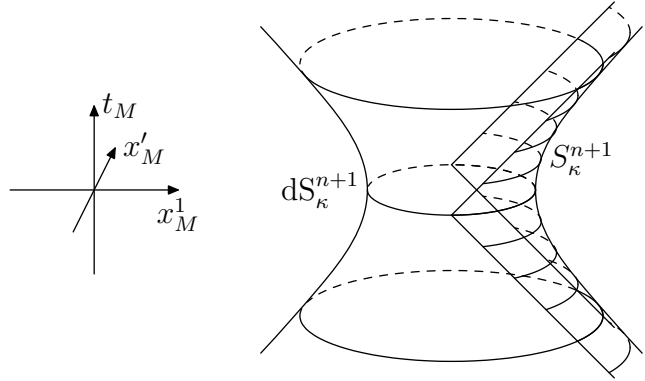}
  \caption{The static patch $S_\kappa^{n+1}$ of de Sitter space as a subset of the embedding of $\dS_\kappa^{n+1}$ into $(n+2)$-dimensional Minkowski space. Also shown are select level sets of $t_M$ within $S_\kappa^{n+1}$.}
  \label{FigdSStatic}
\end{figure}

We introduce static coordinates $(t,\rho,\theta)\in\R\times[0,\kappa^{-1})\times\Sph^{n-1}$ on $S_\kappa^{n+1}$ via
\begin{equation}
\label{EqdSStaticCoordChange1}
\begin{split}
  t_M &= (\kappa^{-2}-\rho^2)^{1/2}\sinh(\kappa t), \\
  x_M^1 &= (\kappa^{-2}-\rho^2)^{1/2}\cosh(\kappa t), \\
  x_M' &= \rho\theta,
\end{split}
\end{equation}
and the de~Sitter metric on $S_\kappa^{n+1}$ takes the well-known form
\begin{equation}
\label{EqdSStaticMetric}
  g_{S_\kappa^{n+1}} \equiv g_{\dS_\kappa^{n+1}}|_{S_\kappa^{n+1}} = -(1-\kappa^2\rho^2)\,dt^2 + (1-\kappa^2\rho^2)^{-1}\,d\rho^2 + \rho^2 g_{\Sph^{n-1}}.
\end{equation}
The singularity of this expression at $\rho=\kappa^{-1}$ is clearly a coordinate singularity since the global de~Sitter metric $g_{\dS_\kappa^{n+1}}$ extends smoothly to $|x_M'|=\kappa^{-1}$ and beyond. Concretely, introduce the Kerr-star type coordinate
\begin{equation}
\label{EqdSStaticTStar}
  t_* = t + \frac{1}{2\kappa}\log(1-\kappa^2\rho^2),
\end{equation}
then $dt=dt_*+\frac{\kappa \rho}{1-\kappa^2\rho^2}d\rho$, so
\begin{align*}
  g_{S_\kappa^{n+1}} &= -(1-\kappa^2\rho^2)\,dt_*^2 - 2\kappa \rho\,dt_*\,d\rho + d\rho^2 + \rho^2\theta^2 \\
   &= -(1-\kappa^2|X|^2)\,dt_*^2 - 2\kappa X\,dt_*\,dX + dX^2,
\end{align*}
where $X=\rho\theta\in\R^n$; this does extend beyond $|X|=\kappa^{-1}$ as a Lorentzian metric. Furthermore, this is closely related to the upper half space model: indeed, with
\begin{equation}
\label{EqdSStaticCoordChange2}
  \tau=\kappa^{-1}e^{-\kappa t_*},\quad x=e^{-\kappa t_*}X,
\end{equation}
we have
\[
  g_{S_\kappa^{n+1}} = \frac{-d\tau^2+dx^2}{\kappa^2\tau^2},
\]
which is the same expression as \eqref{EqdSUHMetric}. (In fact, the coordinate change \eqref{EqdSUHCoordChange} equals the composition of the two coordinate changes \eqref{EqdSStaticCoordChange1} and \eqref{EqdSStaticCoordChange2}.) See also Fig.~\ref{FigdStstar}, whose right panel combines the $\tau$ level sets of Fig.~\ref{FigdSE} with the depiction of the static patch within $E^{n+1}$ in Fig.~\ref{FigdSUH}.

\begin{figure}[!ht]
  \centering
  \includegraphics{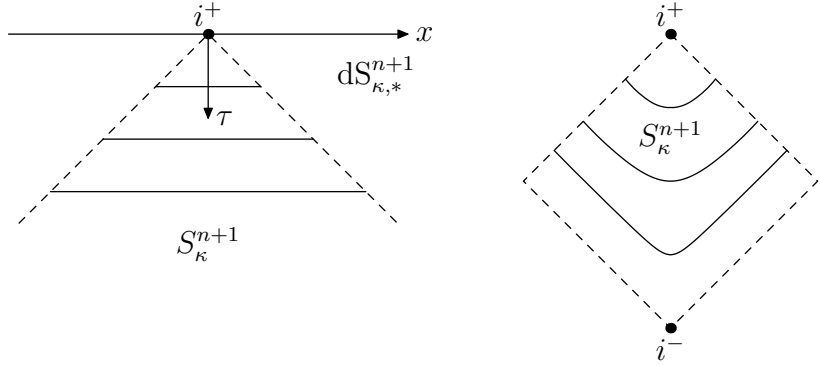}
  \caption{\textit{Left:} the domain $\dS_{\kappa,*}^{n+1}$ in which the coordinates $(\tau,x)$ are valid, together with the static patch $S_\kappa^{n+1}$, bounded by $|x|=\tau$, and level sets of the function $\tau$ (or $t_*$) within $S_\kappa^{n+1}$. \textit{Right:} Penrose diagram of the static patch, together with the same level sets.}
  \label{FigdStstar}
\end{figure}

The coordinates $(t_*,\rho,\theta)$ are valid in the same set $\dS_*^{n+1}$ in which $(\tau,x)$ are valid. Observe that on the subset $\{x_M^1+t_M=0\}\subset\dS_\kappa^{n+1}$, we have $|x_M'|=\kappa^{-1}$; it follows that in $t_M\leq 0$, $\dS_*^{n+1}$ coincides with the static patch corresponding to the point $i^-$, while in $t_M\geq 0$, $\dS_*^{n+1}$ is the complement of the static patch corresponding to the antipodal point of $i^+_0$ as a point on future infinity (that is, $(\pi,-i^+_0)$ in the Einstein universe).

\begin{rmk}
  Writing \eqref{EqdSStaticCoordChange2} as $X=\frac{x}{\kappa\tau}$ exhibits $(\tau,X)$ as coordinates near the interior of the front face of the (homogeneous) blowup of $[0,\infty)_\tau\times\R_x^n$ at $(0,0)$.
\end{rmk}

\subsection{Estimates on resonance widths via general relativity}
\label{s:rela}

We recall that hyperbolic space \eqref{EqHyp} is an Einstein metric, $\Ric(g_\kappa)=-(n-1)\kappa^2 g_\kappa$, with scalar curvature $R_{g_\kappa}=-n(n-1)\kappa^2$. Upon setting $\rho=\kappa^{-1}\tanh(\kappa r)\in(0,\kappa^{-1})$, this becomes
\[
  g_\kappa = \frac{d\rho^2}{(1-\kappa^2\rho^2)^2} + \frac{\rho^2}{1-\kappa^2\rho^2} g_{\Sph^{n-1}},
\]
which is the Klein model of hyperbolic space.

\begin{rmk}
\label{rmk:even}
  The coordinate change $\rho=\frac{2 r_0}{1+\kappa^2 r_0^2}$, $z=r_0\omega$, expresses the hyperbolic metric as
  \[
    g_\kappa = \frac{4}{(1-\kappa^2|z|^2)^2}\,dz^2.
  \]
  For $\kappa=1$, this is an asymptotically hyperbolic metric in the sense explained in \S\ref{s:mercont} if we take e.g.\ $x=1-\kappa^2|z|^2$ as the defining function of the conformal boundary. Note that
  \[
    x^2 = \frac{4 r_0^2}{\rho^2} (1-\kappa^2\rho^2),
  \]
  hence smooth functions on the even compactification are precisely those functions which are smooth in $(1-\kappa^2\rho^2)$.
\end{rmk}

Recall from \eqref{eq:tildegkappa} the static Lorentzian metric $\tilde g_\kappa$ on $\R_t\times\HH_\kappa^n$: this metric is conformal to the static de~Sitter metric \eqref{EqdSStaticMetric}, namely
\begin{equation}
\label{eq:HindS}
  (1-\kappa^2\rho^2)\tilde g_\kappa = g_{S_\kappa^{n+1}}
\end{equation}
upon identifying the coordinate systems $(t,\rho,\theta)$ on $\R_t\times\HH_\kappa^n$ and $S_\kappa^{n+1}$.

Returning to the analysis of scattering resonances on hyperbolic space, we first discuss the case with no obstacle present. Thus, suppose $\tilde v$ is a resonant state of $P_\kappa$,
\[
  \tilde v=(1-\kappa^2\rho^2)^{\frac{n-1}{4}-\frac{i\lambda}{2\kappa}}v,
\]
where is $v$ smooth on the even compactification $(\overline{\HH_\kappa^n})_{\mathrm{even}}$ of $\HH_\kappa^n$ by Theorem~\ref{t:mercont}, that is, $v$ extends to a smooth function of $(1-\kappa^2\rho^2)$ for $0<\rho\leq\kappa^{-1}$ -- see Remark~\ref{rmk:even}. Thus, $\tilde v$ solves
\[
  e^{i\lambda t}\Bigl(\Box_{\tilde g_\kappa} + \Bigl(\frac{n-1}{2}\Bigr)^2\kappa^2\Bigr) e^{-i\lambda t}\tilde v = 0.
\]
Put
\begin{equation}
\label{EqHwtu}
  \tilde u:=(1-\kappa^2\rho^2)^{-\frac{i\lambda}{2\kappa}}e^{-i\lambda t}v = e^{-i\lambda t_*}v,
\end{equation}
where we use the function $t_*$ defined in \eqref{EqdSStaticTStar}; then $\tilde u$ is a smooth function on $S_\kappa^{n+1}$ which extends smoothly across the boundary of $S_\kappa^{n+1}$ in $t\geq 0$, and in fact $\tilde u$ extends smoothly to the region of validity $\dS_{\kappa,*}^{n+1}$ of the coordinates $(t_*,\rho,\theta)$. Moreover, it solves
\begin{equation}
\label{EqHWave}
  (1-\kappa^2\rho^2)^{-1}(1-\kappa^2\rho^2)^{-\frac{n-1}{4}}\Bigl(\Box_{\tilde g_\kappa} + \Bigl(\frac{n-1}{2}\Bigr)^2\kappa^2\Bigr)(1-\kappa^2\rho^2)^{\frac{n-1}{4}}\tilde u = 0,
\end{equation}

\begin{rmk}
  Note that $1-\kappa^2\rho^2=\cosh(\kappa r)^{-2}$, hence we can also write
  \[
    g_{S_\kappa^{n+1}} = \cosh(\kappa r)^{-2}(-dt^2+dr^2) + \kappa^{-2}\tanh(\kappa r)^2 g_{\Sph^{n-1}}.
  \]
  Compare this with Lemma~\ref{l:conj}.
\end{rmk}

Recall now the transformation of a wave operator under conformal transformations: if $(M,g)$ is an $(n+1)$-dimensional Lorentzian manifold, then
\begin{equation}
\label{EqHwConf}
  e^{-2\phi}e^{-\frac{n-1}{2}\phi}\Bigl(\Box_g - \frac{n-1}{4n}R_g\Bigr)e^{\frac{n-1}{2}\phi} = \Box_{e^{2\phi}g} - \frac{n-1}{4n}R_{e^{2\phi}g}.
\end{equation}
Applying this to equation~\eqref{EqHWave}, with $e^{2\phi}=1-\kappa^2\rho^2$, $g=\tilde g_\kappa$, for which we indeed have $\frac{(n-1)R_g}{4 n}=-(\frac{n-1}{2})^2\kappa^2$, we find
\begin{equation}
\label{EqHWave2}
  \Bigl(\Box_{g_{S_\kappa^{n+1}}}-\frac{n^2-1}{4}\kappa^2\Bigr)\tilde u = 0.
\end{equation}

Let now $\cO$ denote a star-shaped obstacle in $\HH_\kappa^n$ with smooth boundary. If $\lambda\in\C$ is a resonance of $P_\kappa$, then an associated resonant state $\tilde v$ on $\HH_\kappa^n$ with Dirichlet boundary conditions on $\pa\cO$ is a function $\tilde v$ as above which in addition satisfies $\tilde v|_{\pa\cO}=0$. Thus, the function $\tilde u$ defined in \eqref{EqHwtu} solves equation \eqref{EqHWave2} and satisfies
\begin{equation}
\label{EqHDirichlet}
  \tilde u|_{\pa\tilde\cO} \equiv 0,\quad \tilde\cO:=\R_{t_*}\times\cO.
\end{equation}
For any non-trivial resonant state $\tilde v$, the function $\tilde u$ must be non-constant on the level sets of $t_*$ in the static patch $S_\kappa^{n+1}=\{\rho<\kappa^{-1}\}$. Thus, in order to obtain a lower bound on $|\Im\lambda|$, it suffices to prove exponential decay (in $t_*$) of spatial derivatives of $\tilde u$ in $S_\kappa^{n+1}$. To state this precisely, we use the coordinates $t_*$ and $X=r\theta\in\R^n$:

\begin{lemm}
\label{LemmaHDecayToRes}
  Suppose $\alpha>0$ is such that for all solutions $\tilde u$ of equation~\eqref{EqHWave2} defined in $S_\kappa^{n+1}\cap\{t_*\geq 0\}$, smooth up to the cosmological horizon $\pa S_\kappa^{n+1}=\{|X|=\kappa^{-1}\}$, and satisfying the Dirichlet boundary condition \eqref{EqHDirichlet}, there exists a constant $C$ such that
  \begin{equation}
  \label{EqHDecayToResEst}
    \int_{|X|<\kappa^{-1}} |\partial_X\tilde u|^2\,dX \leq C e^{-2\alpha t_*},\quad t_*>0.
  \end{equation}
  Then all resonances $\lambda$ of $P_\kappa$ satisfy
  \[
    \Im\lambda\leq-\alpha.
  \]
\end{lemm}

\begin{proof}[Proof of Theorem~\ref{t:1} for all $n\geq 2$]
We will obtain the estimate~\eqref{EqHDecayToResEst} by relating equation~\eqref{EqHWave2} to yet another wave equation via a conformal transformation. Namely, in the coordinates $(\tau,x)\in(0,\infty)\times\R^n$ defined in \eqref{EqdSStaticCoordChange2}, we have $(\kappa\tau)^2 g_{S_\kappa^{n+1}} = g_M := -d\tau^2 + dx^2$, hence the rescaled function $u = (\kappa\tau)^{-\frac{n-1}{2}}\tilde u$ satisfies the equation $\Box_{g_M}u = 0$ with Dirichlet boundary conditions on
\begin{equation}
\label{eq:minkobs}
  \tilde\cO = \Bigl\{ (\tau,x) \colon \frac{x}{\kappa\tau} \in \cO \Bigr\}.
\end{equation}
Note that for $\tilde u$ defined in $S_\kappa^{n+1}\cap\{t_*\geq 0\}$, the function $u$ is defined in $|x|<\tau<\kappa^{-1}$. Notice however that the Cauchy data $(u_0,u_1)$ of $u$ at $\tau=\kappa^{-1}$ can be extended to compactly supported data $(w_0,w_1)$ on $\{\tau=\kappa^{-1},\ |x|<2\tau\}$ whose $H^1$ norm is controlled by a uniform constant times the $H^1$ norm of $(u_0,u_1)$, and the solution $w$ of the Cauchy problem $\Box_{g_M}w=0$ with Cauchy surface $\tau=\kappa^{-1}$ exists (and is smooth) on $\tau^{-1}((0,\kappa^{-1}])$ and equals $u$ in $S_\kappa^{n+1}\cap\{0<\tau\leq\kappa^{-1}\}$, the domain of dependence of $\{|x|<\tau,\ \tau=\kappa^{-1}\}$. See Fig.~\ref{FigHWave}.

\begin{figure}[!ht]
  \centering
  \includegraphics{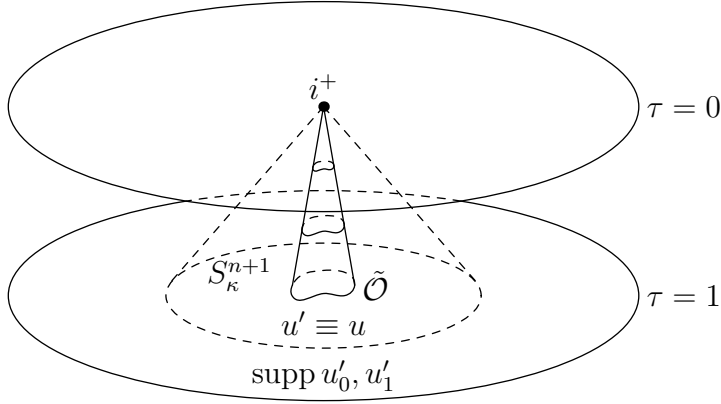}
  \caption{The obstacle $\tilde\cO=\R_{t_*}\times\cO$ in the static de~Sitter patch $S_\kappa^{n+1}$, which itself is embedded in the upper half plane model $\dS_{\kappa,*}^{n+1}$, which in turn is conformally diffeomorphic to a half space $\tau>0$ of Minkowski space with the metric $-d\tau^2+dx^2$.}
  \label{FigHWave}
\end{figure}

Without the obstacle, $w$ would satisfy arbitrary order energy estimates uniformly up to $\tau=0$ and beyond. With the obstacle present, we can only control first order energies when using the future timelike vector field $-\pa_\tau$; note that this vector field points out of $\tilde\cO$ at the boundary $\pa\tilde\cO$ of the obstacle. Since the latter is smooth in $\tau>0$, we have
\[
  \int_{\tau=\tau_0} |\partial_\tau w|^2 + |\partial_x w|^2 \,dx \leq \int_{\tau=1} |\partial_\tau w|^2 + |\partial_x w|^2 \,dx \leq C,\ \ \ \tau_0>0;
\]
the key is that this holds \emph{uniformly} for all $\tau_0>0$. Dropping the $\tau$-derivative on the left, restricting the domain of integration to $|x|<\tau_0$, and using $\pa_x=(\kappa\tau)^{-1}\pa_X$ as well as $dx=(\kappa\tau)^n\,dX$, this gives
\begin{equation}
\label{EqHEstimate}
  C \geq (\kappa\tau)^{n-2}\int_{|X|<\kappa^{-1}} |\partial_X w|^2\,dX = (\kappa\tau)^{-1} \int_{|X|<\kappa^{-1}} |\partial_X\tilde u|^2\,dX.
\end{equation}
Since $\tau=e^{-t_*}$, the estimate~\eqref{EqHDecayToResEst} holds with $\alpha=\kappa/2$, giving the universal lower bound $\kappa/2$ for the resonance width and thus proving Theorem~\ref{t:1}.
\end{proof}

We remark that all resonances with $\Im\lambda=-\kappa/2$ must be semisimple, as otherwise there would be solutions $\tilde u$ with $L^2$ norm of $\tilde u_X$ bounded from below by $e^{-\kappa t_*/2}t_*$, contradicting \eqref{EqHEstimate}.

\begin{rmk}
  The estimate~\eqref{EqHDecayToResEst} is in fact \emph{false} for $\alpha>\kappa/2$; this is related to the fact that $H^1$ is the threshold regularity for radial point estimates at the decay rate $\kappa/2$, see \cite[Proposition~2.1]{HV15}, and says that control of $H^1$ alone is not sufficient for proving a lower bound for resonance widths which is better than $\kappa/2$. Indeed, take $\lambda\in\C$ with $\Im\lambda=-\kappa/2-\eps$, $\eps>0$ small, which is \emph{not} a resonance of $P_\kappa$. Define
  \[
    \tilde v' := (1-\kappa^2\rho^2)^{\frac{n-1}{4}+\frac{i\lambda}{2\kappa}}\chi(1-\kappa^2\rho^2),
  \]
  where $\chi\in\CI([0,1))$, $\chi(0)=1$, is chosen such that $(P_\kappa-\lambda^2)\tilde v'$ vanishes to infinite order at $1-\kappa^2\rho^2=0$. Let then
  \[
    \tilde v = \tilde v' - R_\kappa(\lambda)\bigl((P_\kappa-\lambda^2)\tilde v'\bigr),
  \]
  which solves $(P_\kappa-\lambda^2)\tilde v=0$. Since $R_\kappa(\lambda)$ produces an outgoing function, while $\tilde v'$ is ingoing, we have $\tilde v\neq 0$. Let
  \[
    v=(1-\kappa^2\rho^2)^{-\frac{n-1}{4}+\frac{i\lambda}{2\kappa}}\tilde v,
  \]
  then $v\in H^1((\overline{\HH_\kappa^n})_{\mathrm{even}})$ by our assumption on $\lambda$. The function $\tilde u:=e^{-i\lambda t_*}v$ solves equation~\eqref{EqHWave2}, and satisfies the estimate~\eqref{EqHDecayToResEst} only when $\alpha\leq\kappa/2+\eps$.
\end{rmk}

\subsection{Non-trapping property of star-shaped obstacles}
\label{s:nontrap}

In order to justify the relationship between resonance widths and exponential energy decay for waves in the exterior of star-shaped obstacles, we prove:

\begin{prop}
  Let $\mathcal O\subset\HH^n$ be star-shaped. Then $\HH^n\setminus\mathcal O$ is non-trapping.
\end{prop}
\begin{proof}
  As in the proof of the analogous result in Euclidean space given in~\cite{PetkovStoyanov}, we will identify a quantity which is monotone along suitably rescaled broken geodesics. Suppose $\mathcal O\subset B(0,R)$, $R>0$. As in the reference, it suffices to prove that if $\gamma\colon\R\to\HH^n\setminus\mathcal O$ is a unit speed broken geodesic with $\gamma(0)\in B(0,R)$ all of whose intersections with $\partial\mathcal O$ are \emph{transversal}, then $\gamma(t)\notin B(0,R)$ for all sufficiently large $t>0$. Recall the definition~\eqref{eq:tildegkappa} of the Lorentzian metric $\tilde g_1$, and note that the curve $\tilde\gamma(t):=(t,\gamma(t))$ is a broken null-geodesic on $(\R_t\times(\HH^n\setminus\mathcal O),\tilde g_1)$. Now, images of null-geodesics are invariant under conformal changes of the metric, so let us consider the image of $\tilde\gamma$ as the image of a broken null-geodesic $\gamma_M(t)$ in \emph{Minkowski space} $\R_\tau\times\R^n_x$ with obstacle $\tilde{\mathcal O}$, as discussed around~\eqref{eq:minkobs} and in Figure~\ref{FigHWave}. Writing $\gamma_M(t)=(\tau_M(t),x_M(t))$, we may assume without loss that $\tau_M(0)=-1$, $\tau_M'(0)=1$, and thus $|x_M'(0)|=1$; we then have $|x_M(0)|\leq c$ for some $c<1$. By reparameterizing $\gamma_M$ between any two reflection points, we can arrange that $\tau_M(t)=-1+t$, hence also $|x_M'(t)|=1$ for all $t<1$. If we define $f(t):=x_M(t)\cdot x_M'(t)$ using the Euclidean inner product, then $f'(t)=1$ whenever $\gamma_M(t)\notin\partial\tilde{\mathcal O}$, while at a reflection point $\gamma_M(t_0)$, we have
  \begin{equation}
  \label{eq:refl}
    f(t_0+0)\geq f(t_0-0),
  \end{equation}
  as we will prove momentarily. Therefore,
  \[
    f(t) \geq f(0) + t \geq -c+t \geq \eps > 0
  \]
  for $\eps\in(0,1-c)$ and $t>c+\eps$. Since $|x'_M(t)|=1$, this forces $x_M(t)>\eps$ for such $t$. Therefore, $\gamma_M(t)$ lies beyond the cosmological horizon of the static de~Sitter patch $S_1^{n+1}$ for $t$ close to $1$, which means that $\gamma_M$ is not trapped.

  To prove~\eqref{eq:refl}, we consider (after rescaling) a point $p=(-1,x)\in\partial\tilde{\mathcal O}\subset\R_\tau\times\R_x^n$, $|x|<1$. Within $\{\tau=-1\}\cong\R^n$, denote the outward pointing unit normal of $\mathcal O\cong\tilde{\mathcal O}\cap\{\tau=-1\}$ by $n\in\R^n$, so $T_p\tilde{\mathcal O}$ is spanned by $n^\perp\subset\R^n\subset\R^{1+n}$ and $(-1,x)$. A (Lorentzian) normal vector to $\tilde(\mathcal O)$ at $p$ is thus $\nu:=(-n\cdot x,n)$. Given an inward pointing vector $v_-:=\gamma_M'(t_0-0)=(1,w)$, $|w|=1$, we thus have $0\leq g_M(\nu,(1,w))=-n\cdot(x+w)$. The reflection of $v_-$ is
  \begin{align*}
    v'_+ = v_- - \frac{2 g_M(\nu,v_-)}{g_M(\nu,\nu)}\nu = (1-(n\cdot x)^2)^{-1}\bigl(&1+(n\cdot x)^2+2(n\cdot w)(n\cdot x), \\
    &\quad w-(n\cdot x)^2 w-2(n\cdot(x+w))n \bigr).
  \end{align*}
  Let $v_+=v'_+/\lambda=:(1,w')$, where $\lambda$ is the first component of $v'_+$; then $\gamma_M'(t_0+0)=v_+$, and we need to verify $x\cdot w'\geq x\cdot w$, i.e.\ after clearing denominators and simplifying,
  \[
    (n\cdot x)(1+x\cdot w)\bigl(n\cdot(x+w)\bigr) \leq 0,
  \]
  which indeed holds, since $n\cdot x\geq 0$ since $\mathcal O$ is star-shaped.
\end{proof}

\section{Small obstacles and Euclidean resonances}
\label{s:smalle}

Let $n\geq 3$ be odd. Suppose $\cO\subset\R^n$ is a compact domain with smooth boundary. We can then identify $\cO$ with a smooth domain $\HH^n_\kappa$ via the identification $\HH^n_\kappa\cong\R^n$ of smooth manifolds in \eqref{EqHyp}. Formally taking the limit $\kappa\to 0$, we denote by
$
  g_0 = dr^2 + r^2 h
$ the usual Euclidean metric on $\R^n=:\HH_0^n$. 
We recall that for $\kappa\geq 0$,  the operator $ P_\kappa $ given in
\eqref{EqHypOp} 
is self-adjoint with Dirichlet boundary conditions, that is, with domain $(H^2\cap H^1_0)(\HH_\kappa^n\setminus\cO)$, where we use the metric $g_\kappa$ to define Sobolev spaces. As reviewed in \S\ref{s:mercont}, the resolvent $(P_\kappa-\lambda^2)^{-1}$ admits a meromorphic continuation from $\Im\lambda\gg 0$ to $\C_\lambda$; we denote the set of its poles, counted with multiplicity, by $\Res(\cO,\kappa)$.

In this section we will prove a precise version of Theorem \ref{ThmMain1}:
\begin{thm}
\label{ThmMain}
  We have $\Res(\cO,\kappa)\to\Res(\cO,0)$ locally uniformly, with multiplicities, as $\kappa\to 0$.
More precisely, the set of accumulation 
points of $ \Res ( \cO, \kappa ) $ is contained in $ \Res ( \cO, 0 ) $, and
for any $ K \Subset \CC $  there exist $ r_0 $ and $ \kappa_0 $ such that if $ \lambda_0 \in \Res ( \cO , 0 ) \cap K 
$ has multiplicity $ m $ then for any $ \kappa < \kappa_0 $, 
\[
 \Res ( \cO, \kappa ) \cap D ( \lambda_0 , r_0 ) = \{ \lambda_j ( \kappa ) \}_{j=1}^m , \  \ \lim_{\kappa \to 0 } \lambda_j ( \kappa ) = \lambda_0.
\]
\end{thm}

We begin by computing the kernel of the free resolvent
\[
  R_\kappa^0(\lambda) = \left(-\Delta_{g_\kappa}-\left(\tfrac{n-1}{2}\right)\kappa^2-\lambda^2\right)^{-1}.
\]

\begin{lemm}
\label{LemmaHypFree}
  For fixed $y$, the resolvent kernel $R_\kappa^0(\lambda;x,y)$ of $\HH^n_\kappa$ is $L^1_\loc$ in $x$. It only depends on the geodesic distance $d_\kappa(x,y)$ between $x$ and $y$, and is given explicitly by
  \[
    R_\kappa^0(\lambda;x,y) = -\frac{1}{2 i\lambda}\Bigl(-\frac{1}{2\pi s_\kappa}\pa_r\Bigr)^{\frac{n-1}{2}} e^{i\lambda r} \Bigr|_{r=d_\kappa(x,y)}.
  \]
  In particular, $R_\kappa^0(\lambda)$ is entire in $\lambda$.
\end{lemm}
\begin{proof}
  See \cite[\S 8.6]{Tay}; we present a direct proof, based on induction on $j=(n-1)/2\in\Z_{\geq 0}$. The asserted dependence only on $d_\kappa(x,y)$ follows from the fact that $\HH^n_\kappa$ is a symmetric space. Dropping the subscript $\kappa$, denote
  \[
    f_0(r) := -\frac{e^{i\lambda r}}{2 i\lambda},\quad f_{j+1}(r) := -\frac{1}{2\pi s_\kappa}\pa_r f_j(r).
  \]
  We will identify $f_j$, which is a function on $(0,\infty)_r$, with the function $f_j(d_\kappa(0,x))$, $x\in\HH^n_\kappa$. Since $|f_j(r)| \leq s_\kappa^{1-2 j}$ for $j\geq 1$ and $|d g_\kappa|=s_\kappa^{2 j}dr|dh|$, we have $f_j\in L^1_\loc(\HH^n_\kappa)$.

  Fix $\lambda\in\C$ and $\kappa\geq 0$. Denote $P_j := -\Delta_{g_{\HH^{2 j+1}_\kappa}} - j^2\kappa^2 - \lambda^2$, and write
  \[
    Q_j := s_\kappa^{-2 j}D_r s_\kappa^{2 j}D_r - j^2\kappa^2 - \lambda^2
  \]
  for its radial part, which is an operator on $(0,\infty)$. Now for $j=0$, $f_0$ indeed solves $P_0 f_0=\delta_0$. For the inductive step, we note the intertwining relation
  \[
    Q_{j+1}\circ s_\kappa^{-1}\pa_r = s_\kappa^{-1}\pa_r\circ Q_j,
  \]
  which is verified by direct calculation. In verifying that $P_{j+1}f_{j+1}=\delta_0$, we note that, due to the spherical symmetry of $f_{j+1}$, it suffices to check this for radial test functions $\varphi\in\CIc(\HH_\kappa^{2(j+1)+1})$; but for such $\varphi$, we compute the distributional pairing
  \begin{align*}
    \la P_{j+1}s_\kappa^{-1}\pa_r f_j, \varphi \ra_{L^2(\HH^{2(j+1)+1}_\kappa)} & = \vol(\Sph^{2 j+2})\la Q_{j+1}s_\kappa^{-1}\pa_r f_j, \varphi \ra_{L^2(\R_+;s_\kappa^{2(j+1)}dr)} \\
      &= \vol(\Sph^{2 j+2}) \la s_\kappa^{-1}\pa_r Q_j f_j,\varphi\ra_{L^2(\R_+;s_\kappa^{2(j+1)}dr)} \\
      &= -\vol(\Sph^{2 j+2})\la Q_j f_j, s_\kappa^{-2 j}\pa_r s_\kappa^{2 j+1}\varphi \ra_{L^2(\R_+;s_\kappa^{2 j}dr)} \\
      &= -\frac{\vol(\Sph^{2 j+2})}{\vol(\Sph^{2 j})} \la P_j f_j, (2 j+1)\cosh(\kappa r)\varphi + s_\kappa\pa_r\varphi\ra_{L^2(\HH^{2 j+1}_\kappa)} \\
      &= -\frac{(2 j+1)\vol(\Sph^{2 j+2})}{\vol(\Sph^{2 j})} \varphi(0) = -2\pi\varphi(0).
  \end{align*}
  The proof is complete.
\end{proof}

We will use a direct construction of the meromorphic continuation \eqref{EqHypResolvent} using layer potentials. This is convenient for the control of multiplicities. As preparation for this, we study the operator $P^i_\kappa$, defined by the same expression \eqref{EqHypOp}, but now in the \emph{interior} of $\cO$:  $P^i_\kappa$ is self-adjoint with domain $(H^2\cap H^1_0)(\cO)$.

\begin{lemm}
\label{LemmaHypPos}
  We have $P_\kappa\geq 0$ and $P_\kappa^i\geq 0$.
\end{lemm}

For Neumann boundary conditions, $P_\kappa^i$ is \emph{not} non-negative for $\kappa>0$, as then $\la P_\kappa^i 1,1\ra=-(n-1)^2\kappa^2\vol_{g_\kappa}(\cO)/4<0$.

\begin{proof}[Proof of Lemma~\ref{LemmaHypPos}]
  We use the upper half space model of hyperbolic space $(\HH^n_\kappa,g_\kappa)\cong((0,\infty)_x\times\R^{n-1}_y, \frac{dx^2+dy^2}{\kappa^2 x^2})$. For $u\in\CIc(\cO)$, we then have
  \begin{align*}
    \la P_\kappa^i u,u\ra &= \int_\cO |\nabla_{g_\kappa} u|^2 - \bigl(\tfrac{n-1}{2}\bigr)^2\kappa^2 |u|^2\,dg_\kappa \\
      &= \kappa^2 \iint_\cO x^{2-n}|u_x|^2 - x^{-n}\bigl(\tfrac{n-1}{2}\bigr)^2|u|^2 + x^{2-n}|u_y|^2\,dx\,dy \\
      &= \kappa^2 \iint_\cO x\bigl|(x^{-\frac{n-1}{2}}u)_x\bigr|^2 + x^{2-n}|u_y|^2 + \bigl(\tfrac{n-1}{2}x^{1-n}|u|^2\bigr)_x\,dx\,dy \geq 0,
  \end{align*}
  where in the last step we used the vanishing of $u$ on $\pa\cO$. The argument for $P_\kappa$ is the same.
\end{proof}

By the spectral theorem, the non-negativity of $P_\kappa$ implies that $R_\kappa(\lambda)$ is holomorphic in $\Im\lambda>0$ as an operator on $L^2(\HH^n_\kappa)$.

\begin{lemm}
\label{LemmaAbsenceReal}
  The meromorphically continued resolvent $R_\kappa(\lambda)$ is regular for $\lambda\in\R$ if $\kappa=0$, and for $0\neq\lambda\in\R$ if $\kappa>0$.
\end{lemm}
\begin{proof}
  For $\kappa=0$, this is a standard consequence of the fact that putative resonant states are outgoing, that is, they satisfy the Sommerfeld radiation condition. For $\lambda\neq 0$ then, Rellich's theorem, \cite[Theorem~3.32]{dizzy}, yields the result, while for $\lambda=0$, one applies the maximum principle, see \cite[Theorem~4.19]{dizzy}. For $\kappa>0$ and $\lambda\neq 0$, a boundary pairing argument together with unique continuation at the conformal boundary of $\HH_\kappa^n$ yields the result -- see \cite[\S3.2]{HiVaForms} and \cite{Ma:uniq}.
\end{proof}

\begin{rmk}
  For star-shaped obstacles in $\HH^n_\kappa$, $\kappa>0$, one can deal with all real $\lambda$ at once by observing that a non-trivial resonant state with real frequency would give rise to a stationary or polynomially growing solution $\tilde u$ of the Klein--Gordon equation $\bigl(\Box_{g_{\dS_\kappa^{n+1}}}+\tfrac{n^2-1}{4}\bigr)\tilde u=0$ on static de~Sitter space, with $\tilde u|_{\R\times\pa\cO}=0$, which is smooth up to (and across) the cosmological horizon of $\dS_\kappa^{n+1}$. The energy estimates proved in \S\ref{s:rela} show however that non-trivial such $\tilde u$ do not exist.
\end{rmk}

Our proof of Theorem~\ref{ThmMain} implies the absence of a resonance at $0$ for small $\kappa>0$ (depending on the obstacle). In order to analyze resonances in $\Im\lambda<0$ in an effective manner, we consider the closely related boundary value problem
\begin{equation}
\label{EqHypBVP}
  \begin{cases}
    \bigl(-\Delta_{g_\kappa}-\bigl(\tfrac{n-1}{2}\bigr)^2\kappa^2 - \lambda^2\bigr)u = 0 & \text{in}\ \HH_\kappa^n\setminus\cO, \\
    u|_{\pa\cO} = f & \text{on}\ \pa\cO,
  \end{cases}
\end{equation}
with $f\in H^{3/2}(\pa\cO)$ given, and where we seek an outgoing solution $u\in H^2_\loc(\HH_\kappa^n\setminus\cO)$. For $\Im\lambda>0$, this means finding a solution $u\in L^2(\HH_\kappa^n\setminus\cO)$, which is given by
\begin{equation}
\label{EqHypBVPbyForcing}
  u = \cB_\kappa(\lambda)f := E f - R_\kappa(\lambda)(-\Delta_{g_\kappa}-\bigl(\tfrac{n-1}{2}\bigr)^2\kappa^2-\lambda^2\bigr)E f,
\end{equation}
where $E\colon H^{3/2}(\pa\cO)\to H^2_\comp(\HH^n_\kappa\setminus\cO)$ is a continuous extension operator. Since $R_\kappa(\lambda)$ is meromorphic, equation~\eqref{EqHypBVPbyForcing} provides the meromorphic continuation of $\cB_\kappa(\lambda)$ to the complex plane in $\lambda$. On the other hand, one can reconstruct $R_\kappa(\lambda)$ from $\cB_\kappa(\lambda)$:

\begin{lemm}
\label{LemmaHypRelation}
  We have
  \begin{equation}
  \label{EqHypForcingbyBVP}
    R_\kappa(\lambda;x,y) = R_\kappa^0(\lambda;x,y) - \cB_\kappa(\lambda)\bigl(R_\kappa^0(\lambda;\cdot,y)|_{\pa\cO}\bigr).
  \end{equation}
\end{lemm}
\begin{proof}
  Applying the operator $-\Delta_{g_\kappa}-\bigl(\tfrac{n-1}{2}\bigr)^2\kappa^2-\lambda^2$ to either side yields $\delta_y(x)$. Moreover, for $\Im\lambda>0$, multiplying either side with $f(y)$, $f\in\CIc(\HH^n_\kappa\setminus\cO)$, and integrating over $y$ gives two $L^2$ solutions $u_L$ and $u_R$ of $P_\kappa u=f$, $u|_{\pa\cO}=0$; but by the spectral theorem, we must have $u_L=u_R$. This establishes \eqref{EqHypForcingbyBVP} for $\Im\lambda>0$; for general $\lambda\in\C$ it then follows by meromorphic continuation.
\end{proof}

Defining the multiplicity of a resonance $\lambda$ of $\cB_\kappa$ as
\[
  m^\cB_\kappa(\lambda) := \dim\Bigl[\Bigl(\oint_\lambda \cB_\kappa(\zeta)\,d\zeta\Bigr)(H^{3/2}(\pa\cO))\Bigr],
\]
we conclude that
\begin{equation}
\label{eq:samemult}
  m_\kappa(\lambda)=m^\cB_\kappa(\lambda),\ \ \ \lambda\in\C.
\end{equation}
In fact, equation~\eqref{EqHypBVPbyForcing} implies $m^\cB_\kappa(\lambda)\leq m_\kappa(\lambda)$, while equation~\eqref{EqHypForcingbyBVP} implies the reverse inequality. In order to study $\cB_\kappa(\lambda)$, we introduce the single layer potential
\[
  \Sl_\kappa(\lambda)f(x) := \int_{\pa\cO} R_\kappa^0(\lambda;x,y)f(y)\,d\sigma_\kappa(y),\ x\in\HH_\kappa^n\setminus\pa\cO,
\]
where $d\sigma_\kappa$ is the surface measure on $\pa\cO$ induced by the volume form $d\mathrm{vol}_{g_\kappa}$. Denote by $\pa_\nu$ the normal vector field of $\pa\cO$ pointing into $\cO$, and for a function $u$ on $\HH^n_\kappa$ for which $u|_{\cO}$ and $u|_{\HH^n_\kappa\setminus\cO}$ are smooth up to $\pa\cO$, denote by $u_+$, resp.\ $u_-$, the limits of $u$ to $\pa\cO$ from $\HH^n_\kappa\setminus\cO$, resp.\ $\cO$. We then recall the formul\ae{}
\[
  (\Sl_\kappa(\lambda)f)_\pm = G_\kappa(\lambda)f,\quad G_\kappa(\lambda)f(x) := \int_{\pa\cO} R_\kappa^0(\lambda;x,y)f(y)\,d\sigma_\kappa(y),\ x\in\pa\cO,
\]
and
\begin{gather*}
  (\pa_\nu\Sl_\kappa(\lambda)f)_\pm = \frac{1}{2}(\mp f+N^\sharp_\kappa(\lambda)f), \\
  N^\sharp_\kappa(\lambda)f(x) := 2\int_{\pa\cO} \pa_{\nu_x}R_\kappa^0(\lambda;x,y)f(y)\,d\sigma_\kappa(y),\ x\in\pa\cO;
\end{gather*}
moreover, $G_\kappa(\lambda),N^\sharp_\kappa(\lambda)\in\Psi^{-1}(\pa\cO)$ are entire in $\lambda$, where $\Psi^s(\pa\cO)$ denotes the space of pseudodifferential operators of order $s\in\R$ on the closed manifold $\pa\cO$ \cite[\S18.1]{H3}. The principal symbol of $G_\kappa(\lambda)$ is given by $\bigl|g_\kappa|_x(\xi,\xi)\bigr|^{-1/2}$, $\xi\in T^*_x\pa\cO$, in particular it is independent of $\lambda$. We note some basic properties:

\begin{lemm}
  $G_\kappa(\lambda)$ is injective for $\Im\lambda>0$, and for $\lambda\in\R\setminus\{0\}$ for which $\lambda^2$ is not an eigenvalue of the interior Dirichlet problem $(P_\kappa^i-\lambda^2)u=0$. Furthermore, 
  $$\Sl_\kappa(\lambda)\colon H^{3/2}(\pa\cO)\to L^2_\loc(\HH^n_\kappa\setminus\cO)$$
   is injective for $\lambda\not\in\R$.
\end{lemm}
\begin{proof}
  This is proved for $\R^3$ in \cite[\S9.7]{Tay}; we give the proof in general for completeness, in particular highlighting the use of the Dirichlet (rather than Neumann) boundary condition. Suppose $G_\kappa(\lambda)g=0$, $\Im\lambda\geq 0$, $\lambda\neq 0$, then $u:=\Sl_\kappa(\lambda)g$, defined on $\HH^n_\kappa\setminus\pa\cO$, solves the exterior problem \eqref{EqHypBVP} with $f=0$, hence $u\equiv 0$ outside $\cO$. Therefore, the restriction $u^i:=u|_\cO$ to the interior of the obstacle solves the Dirichlet problem $(P_\kappa^i-\lambda^2)u^i=0$, with Neumann data
  \[
    \pa_\nu u^i = (\pa_\nu u)_- - (\pa_\nu u)_+ = g.
  \]
  For $\Im\lambda>0$, Lemma~\ref{LemmaHypPos} implies $u^i\equiv 0$, hence $g=0$; for real $\lambda$ on the other hand, if $\lambda^2$ is not an eigenvalue of the interior Dirichlet problem, then $u^i\equiv 0$ as well.

  To prove the final claim, suppose $\Sl_\kappa(\lambda)f=0$ outside $\cO$, then $u^i:=\Sl_\kappa(\lambda)f\in H^2(\cO)$ solves $(P_\kappa^i-\lambda^2)u^i=0$ with vanishing Dirichlet data. Since $\lambda\not\in\R$, Lemma~\ref{LemmaHypPos} implies $u^i\equiv 0$, therefore $f = \pa_\nu u^i=0$, as desired.
\end{proof}

Moreover, $G_\kappa(\lambda)$ is self-adjoint for real $\lambda$, hence by ellipticity it is Fredholm with index $0$ as a map $H^s(\pa\cO)\to H^{s+1}(\pa\cO)$ for all $s\in\R$. Fix $\lambda_0\in\R$ such that $G_\kappa(\lambda_0)$ is injective, hence invertible, then formula
\[
  G_\kappa(\lambda)^{-1} = G_\kappa(\lambda_0)^{-1}(I + \Gamma_\kappa(\lambda)G_\kappa(\lambda_0)^{-1})^{-1},\quad \Gamma_\kappa(\lambda):=G_\kappa(\lambda)-G_\kappa(\lambda_0)\in\Psi^{-2}(\pa\cO),
\]
with $\Gamma_\kappa(\lambda)G_\kappa(\lambda_0)^{-1}\in\Psi^{-1}(\pa\cO)$ entire, gives the meromorphic continuation of
\[
  G_\kappa(\lambda)^{-1}\colon H^s(\pa\cO)\to H^{s+1}(\pa\cO)
\]
from $\Im\lambda>0$ to the complex plane; $G_\kappa(\lambda)^{-1}$ has poles of finite order, and the operators in the Laurent series at a pole have finite rank. Then
\begin{equation}
\label{EqHypBVPCont}
  \cB_\kappa(\lambda) = \Sl_\kappa(\lambda)G_\kappa(\lambda)^{-1} \colon H^{1/2}(\pa\cO)\to H^2_\loc(\HH^n_\kappa\setminus\cO)
\end{equation}
furnishes a direct way of meromorphically continuing $\cB_\kappa(\lambda)$. (By Lemma~\ref{LemmaAbsenceReal}, the poles of $G_\kappa(\lambda)$ in the case that $\lambda^2$ is an interior Dirichlet eigenvalue do not give rise to poles of $\cB_\kappa(\lambda)$.) Moreover, the set of poles of $\cB_\kappa(\lambda)$ agrees in $\Im\lambda<0$ with the set of poles of $G_\kappa(\lambda)^{-1}$. The crucial fact is then:

\begin{prop}
\label{PropHypBVPPoles}
  For a resonance $\lambda\in\C\setminus\R$, we have
  \[
    m^\cB_\kappa(\lambda) = m^G_\kappa(\lambda) := \tr\frac{1}{2\pi i}\oint_\lambda \pa_\lambda G_\kappa(\zeta)\,G_\kappa(\zeta)^{-1}\,d\zeta,
  \]
  where we integrate along a small circle around $\lambda$, oriented counter-clockwise, which does not intersect the real line and does not contain any other resonances.
\end{prop}

In order to prove this, we first give more general formul\ae{} for $m_\kappa(\lambda)$ and $m_\kappa^\cB(\lambda)$ -- see also \cite[\S5.1.1]{HiVaStab}.

\begin{lemm}
\label{LemmaHypFormula}
  For $\lambda\neq 0$, we have
  \begin{equation}
  \label{EqHypFormulaForc}
  \begin{split}
    m_\kappa(\lambda) = \dim \bigl\{&\Res_{\zeta=\lambda} e^{-i\zeta t}R_\kappa(\zeta)f(\zeta) \colon \\
    &\qquad f(\zeta)\ \text{holomorphic with values in}\ L^2_\comp(\HH^n_\kappa\setminus\cO)\bigr\},
  \end{split}
  \end{equation}
  where the space on the right hand side is a subspace of $L^2_\loc(\R_t\times(\HH^n_\kappa\setminus\cO))$. Similarly,
  \begin{equation}
  \label{EqHypFormulaBVP}
  \begin{split}
    m_\kappa^\cB(\lambda) = \dim \bigl\{&\Res_{\zeta=\lambda} e^{-i\zeta t}\cB_\kappa(\zeta)f(\zeta) \colon \\
    &\qquad f(\zeta)\ \text{holomorphic with values in}\ H^{3/2}(\pa\cO)\bigr\}.
  \end{split}
  \end{equation}
\end{lemm}

\begin{rmk}
  These two formulas describe the multiplicity of a resonance $\lambda$ as the dimension of the space of generalized mode solutions, with frequency $\lambda$, of the forward problem for
  \[
    \begin{cases}
      (D_t^2 - P_\kappa)\tilde u = f \in \CIc(\R_t;L^2_\comp(\HH^n_\kappa\setminus\cO)), \\
      \tilde u|_{\R_t\times\pa\cO} = 0,
    \end{cases}
  \]
  in the case of \eqref{EqHypFormulaForc}, and of the forward problem for
  \[
    \begin{cases}
      (D_t^2 - P_\kappa)\tilde u = 0, \\
      \tilde u|_{\R_t\times\pa\cO} = f \in \CIc(\R_t;H^{3/2}(\pa\cO)),
    \end{cases}
  \]
  in the case of \eqref{EqHypFormulaBVP}; the connection is via the Fourier transform in $t$, with $\lambda$ the Fourier dual variable.
\end{rmk}

\begin{proof}[Proof of Lemma~\ref{LemmaHypFormula}]
  Denoting the right hand sides of equations~\eqref{EqHypFormulaForc} and \eqref{EqHypFormulaBVP} by $\tilde m_\kappa(\lambda)$ and $\tilde m^\cB_\kappa(\lambda)$, respectively, we note that the formulas~\eqref{EqHypBVPbyForcing} and \eqref{EqHypForcingbyBVP} imply $\tilde m_\kappa(\lambda)=\tilde m^\cB_\kappa(\lambda)$. In view of \eqref{eq:samemult}, it therefore suffices to prove $m_\kappa(\lambda)=\tilde m_\kappa(\lambda)$. The inequality $m_\kappa(\lambda)\leq\tilde m_\kappa(\lambda)$ is trivial; if $R_\kappa(\lambda)$ were a general finite-meromorphic operator family, the reverse inequality would in general be false. The key here is the special structure of $R_\kappa(\lambda)$ as the meromorphic continuation of the \emph{spectral family of a fixed operator}, see \cite[Theorem~4.7]{dizzy}, which holds in great generality:
  \[
    R_\kappa(\lambda) = \sum_{j=1}^{M_\lambda} \frac{(P_\kappa-\lambda^2)^{j-1}\Pi}{(\zeta^2-\lambda^2)^j} + A(\zeta),
  \]
  with $A$ holomorphic near $\zeta=\lambda$, and $\Pi\colon L^2_\comp(\HH^n_\kappa\setminus\cO)\to L^2_\loc(\HH^n_\kappa\setminus\cO)$ a finite rank operator. Moreover, $P_\kappa-\lambda^2\colon\ran\Pi\to\ran\Pi$, and $(P_\kappa-\lambda^2)^{M_\lambda}\Pi=0$.
  
  Pick a finite-dimensional vector space $V\subset L^2_\comp$ such that $\Pi\colon V\to\ran\Pi$ isomorphically. Identifying $\ran\Pi$ with $V$ via $\Pi|_V$ and choosing a basis of $V$, $\Pi$ is an $M\times M$ matrix, with $M=\rank\Pi$, and $N:=(\Pi|_V)^{-1}(P_\kappa-\lambda^2)\Pi|_V$ is nilpotent. We note that $m_\kappa(\lambda)=\rank\Pi$; this follows from
  \[
    \frac{1}{2\pi i}\oint_\lambda R_\kappa(\zeta)\,d\zeta = (2\lambda)^{-1}\Id + \sum_{j=1}^{M_\lambda-1} \frac{(-1)^j(2j)!}{j!^2(2\lambda)^{2j+1}}N^j
  \]
  on $V$, and the invertibility of operators, such as the one appearing on the right hand side, which differ from the identity by a nilpotent operator.
  
  Expanding $f(\zeta)$ in \eqref{EqHypFormulaForc} in Taylor series in $\zeta^2-\lambda^2$ around $\zeta=\lambda$, the statement of the lemma is reduced to the linear algebra problem to show that
  \[
    M = \dim\biggl\{ \Res_{\zeta=\lambda} \sum_{0\leq\ell<j\leq M_\lambda} e^{-i\zeta t}(\zeta^2-\lambda^2)^{\ell-j}N^j f_\ell \colon f_0,\ldots,f_{M_\lambda-1}\in \C^M \biggr\},
  \]
  with $N$ a nilpotent element of $\C^{M\times M}$, $N^{M_\lambda}=0$. It suffices to show this when $N$ is a single nilpotent Jordan block. But when (abusing notation) $N$ is a nilpotent $M\times M$ Jordan block, and when $f_j=(f_{j,p})_{p=0,\ldots,M-1}$, then the space of vectors of the form
  \begin{align*}
    \sum_{0\leq\ell<j\leq M}&\frac{1}{(j-\ell-1)!}\pa_\zeta^{j-\ell-1}\frac{e^{-i\zeta t}}{(\zeta+\lambda)^{j-\ell}}\biggr|_{\zeta=\lambda} N^j f_\ell \\
    &= \sum_{\ell=0}^{M-1} \frac{1}{\ell!}\pa_\zeta^\ell\frac{e^{-i\zeta t}}{(\zeta+\lambda)^{\ell+1}}\biggr|_{\zeta=\lambda}\biggl(\sum_{j=0}^{M-1-\ell} N^{\ell+j} f_j\biggr)
  \end{align*}
  has the same dimension as the space of $M$-tuples of vectors in $\C^M$
  \[
    \biggl(\sum_{j=0}^{M-1-\ell} N^{\ell+j} f_j\biggr)_{\ell=0,\ldots,M-1} = \biggl( \sum_{q=0}^{M-1-(\ell+p)} f_{q,q+\ell+p} \biggr)_{\ell,p=0,\ldots,M-1},
  \]
  which is the space of $M\times M$ Hankel matrices, and this space is $M$-dimensional, finishing the proof.
\end{proof}

Using the characterization~\eqref{EqHypFormulaBVP}, we now prove Proposition~\ref{PropHypBVPPoles}:

\begin{proof}[Proof of Proposition~\ref{PropHypBVPPoles}]
  Putting $G_\kappa(\zeta)$ near a resonance $\lambda$, $\Im\lambda<0$, into a normal form, see \cite[Theorem~C.7]{dizzy}, it suffices to prove the following abstract statement: if
  \[
    G(\zeta)=\sum_{j=1}^M (\zeta-\lambda)^j \Pi_j + \Bigl(\Id-\sum_{j=1}^M\Pi_j\Bigr)
  \]
  with $M\geq 1$, the $\Pi_j$ finite rank projections, $\Pi_M\neq 0$, $\Pi_j\Pi_k=0$ for $j\neq k$, is a holomorphic family of Fredholm operators acting on a Banach space $X$, and $S(\zeta)\colon X\to Y$ is a holomorphic family of \emph{injective} operators from $X$ into a Fr\'echet space $Y$, then
  \begin{align*}
    W &:= \dim \bigl\{\Res_{\zeta=\lambda} e^{-i\zeta t}S(\zeta)G(\zeta)^{-1}f(\zeta) \colon f(\zeta)\ \text{holomorphic with values in}\ X\bigr\} \\
    &= \frac{1}{2\pi i}\tr\oint_\lambda \pa_\zeta G(\zeta)\,G(\zeta)^{-1}\,d\zeta.
  \end{align*}
  By direct computation, the right hand side is equal to $\sum_{j=1}^M j\rank\Pi_j$. If we denote by
  \[
    W_k := W \cap \biggl\{ \sum_{j=1}^k e^{-i\lambda t}t^{j-1} f_j\colon f_j\in X \biggr\}, \ \ k=1,\ldots,M,
  \]
  the space of all generalized mode solutions of $(D_t^2-P_\kappa)\tilde u=0$ with frequency $\lambda$ for which the highest power of $t$ is at most $t^{k-1}$, it therefore suffices to show
  \begin{equation}
  \label{EqHypBVPPolesQuot}
    \dim W_k/W_{k-1} = \sum_{j\geq k} \rank\Pi_j,
  \end{equation}
  To see this, expand $S(\zeta)=\sum_{j\geq 0}(\zeta-\lambda)^j S_j$, and note that, for $f_\ell\in X$,
  \begin{equation}
  \label{EqHypBVPPolesQuotWhy}
    \Res_{\zeta=\lambda}e^{-i\zeta t}S(\zeta)G(\zeta)^{-1}\sum_{\ell<M} f_\ell(\zeta-\lambda)^\ell \\
  \end{equation}
  lies in $W_M\setminus W_{M-1}$ unless $\Pi_M f_0=0$, due to the injectivity of $S_0=S(\lambda)$; in this case, the coefficient of $e^{-i\lambda t}(-i t)^{M-2}/(M-2)!$ equals
  \[
    S_0(\Pi_{M-1}f_0+\Pi_M f_1) + S_1\Pi_M f_0 = S_0(\Pi_{M-1}f_0 + \Pi_M f_1),
  \]
  which is non-zero unless $\Pi_{M-1}f_0=0$ and $\Pi_M f_1=0$; and so forth. In general, the generalized mode \eqref{EqHypBVPPolesQuotWhy} lies in $W_M\setminus W_k$ unless
  \[
    \Pi_{k+1+j+\ell}f_j = 0,\ \ 0\leq j<M-k,\ 0\leq\ell\leq M-(k+1+j),
  \]
  and it lies in $W_{k-1}$ if and only if this holds true for $k$ replaced by $k-1$. In other words, the map
  \begin{align*}
    &
   \bigoplus_{j=k}^M \ran\Pi_j \ni \sum_{j=k}^M \Pi_j f_{j-k} 
   \mapsto \biggr[\Res_{\zeta=\lambda}\biggl(e^{-i\zeta t}S(\zeta)G(\zeta)^{-1}\sum_{\ell=0}^{M-k} f_\ell(\zeta-\lambda)^\ell\biggr)\biggr] \in W_k/W_{k-1}
  \end{align*}
  is an isomorphism. This proves \eqref{EqHypBVPPolesQuot}, and hence the proposition.
\end{proof}

\begin{proof}[Proof of Theorem~\ref{ThmMain}]
  Let us fix a precompact open set $\Lambda\subset\C$ with smooth boundary such that $\Res(\cO,0)\cap\pa\Lambda=\emptyset$. We will show that
  \begin{equation}
  \label{EqMainConv}
    \sum_{\lambda\in\Lambda\cap\Res(\cO,\kappa)} m_\kappa(\lambda) = \sum_{\lambda\in\Lambda\cap\Res(\cO,0)} m_0(\lambda)
  \end{equation}
  for small $0\leq\kappa<\kappa_0$. This suffices to prove the theorem; indeed, to show that the resonances of $P_\kappa$ in a precompact open set $\Lambda'\subset\C$ with $\Res(\cO,0)\cap\pa\Lambda'=\emptyset$ are $\eps$-close to those of $P_0$ for $\kappa$ small (depending on $\Lambda'$ and $\eps$), denote $\Res(\cO,0)\cap\Lambda'=\{\lambda_1,\ldots,\lambda_N\}$ ($N\geq 0$); one then applies \eqref{EqMainConv} to the sets $\Lambda_j:=\{\lambda\in\C\colon |\lambda-\lambda_j|<\eps'\}$, with $\eps'\in(0,\eps)$ chosen such that $|\lambda_j-\lambda_k|>\eps'$ for all $j\neq k$; this shows that $\Lambda_j$ contains $m_0(\lambda_j)$ resonances of $P_\kappa$, counted with multiplicity, for $\kappa$ small. On the other hand, applying \eqref{EqMainConv} to the complement $\Lambda_c:=\{\lambda\in\Lambda'\colon|\lambda-\lambda_j|>\eps'/2,\ j=1,\ldots,N\}$ shows that $P_\kappa$ has no resonances in $\Lambda_c$ either for small $\kappa$, as desired.

  As a preliminary step towards \eqref{EqMainConv}, we show:
  \begin{equation}
  \label{EqMainConvReal}
    \parbox{0.7\textwidth}{There exists an open neighborhood $\cU\supset\R$ which contains no resonances of $P_\kappa$ for all $0\leq\kappa<\kappa_0$, $\kappa_0$ small.}
  \end{equation}
  The proof of this relies on a slight modification of the construction~\eqref{EqHypBVPCont}. Namely, we use the double layer potential
  \[
    \Dl_\kappa(\lambda)g(x) := \int_{\pa\cO} \pa_{\nu_y}R_\kappa^0(\lambda;x,y)g(y)\,d\sigma_\kappa(y),\ \ x\in\HH^n_\kappa\setminus\pa\cO,
  \]
  which satisfies
  \begin{align*}
    (\Dl_\kappa(\lambda)g)_\pm &= \frac{1}{2}(\pm g+N_\kappa(\lambda)g), \\
      &\qquad N_\kappa(\lambda)g(x) := 2\int_{\pa\cO} \pa_{\nu_y}R_\kappa^0(\lambda;x,y)g(y)\,d\sigma_\kappa(y),\ \ x\in\pa\cO,
  \end{align*}
  with $N_\kappa(\lambda)\in\Psi^{-1}(\pa\cO)$, and $(\pa_\nu\Dl_\kappa(\lambda)g)_+=(\pa_\nu\Dl_\kappa(\lambda)g)_-$. In order to solve the outgoing boundary value problem~\eqref{EqHypBVP}, we make the new ansatz
  \begin{equation}
  \label{EqMainDlAnsatz}
    u = (i\Sl_\kappa(\lambda)+\Dl_\kappa(\lambda))g,
  \end{equation}
  which satisfies the boundary condition provided $(I+N_\kappa(\lambda)+2 i G_\kappa(\lambda))g=f$. Since the operator $I+N_\kappa(\lambda)+2 i G_\kappa(\lambda)\colon H^s(\pa\cO)\to H^s(\pa\cO)$ is Fredholm with index $0$, we conclude that this is solvable provided this operator is injective. Consider $\lambda\in\R$. If $g$ is an element of the kernel, then $u$, defined as in \eqref{EqMainDlAnsatz}, satisfies $u_+=0$ and $(P_\kappa-\lambda^2)u=0$ in $\HH^n_\kappa\setminus\cO$, hence $u\equiv 0$ there if $\kappa=0$, or if $\kappa>0$ and $\lambda\in\R\setminus\{0\}$, and we conclude that in these cases
  \begin{align*}
    u_- &= i G_\kappa(\lambda)g + \frac{1}{2}(-I+N_\kappa(\lambda))g = -g, \\
    \pa_\nu u_- &= i\bigl((\pa_\nu\Sl_\kappa(\lambda)g)_- -(\pa_\nu\Sl_\kappa(\lambda)g)_+\bigr) = i g.
  \end{align*}
  Thus, integrating over $\cO$, we have
  \[
    0 = \Im\big\la(\Delta_{g_\kappa}-\bigl(\tfrac{n-1}{2}\bigr)^2\kappa^2-\lambda^2)u,u\big\ra = \frac{1}{2 i}\int_{\pa\cO} \pa_\nu u\ol u-u\ol{\pa_\nu u}\,d\sigma_\kappa = -\int_{\pa\cO} |g|^2\,d\sigma_\kappa,
  \]
  hence $g=0$, proving injectivity. Therefore, we can write
  \begin{equation}
  \label{EqMainAltB}
    \cB_\kappa(\lambda) = (i\Sl_\kappa(\lambda)+\Dl_\kappa(\lambda))(I+N_\kappa(\lambda)+2 i G_\kappa(\lambda))^{-1},
  \end{equation}
  which we have just shown is regular for $\lambda\in\R$ if $\kappa=0$, and $0\neq\lambda\in\R$ if $\kappa>0$. From the expression \eqref{EqMainAltB} and using Lemma~\ref{LemmaHypFree}, one sees that the regularity of $\cB_0(\lambda)$ at $\lambda=0$ implies that of $\cB_\kappa(\lambda)$ there when $\kappa>0$ is sufficiently small. Hence, $\cB_\kappa(\lambda)$ is regular for all $\lambda\in\R$ for sufficiently small $\kappa$. A simple continuity argument proves \eqref{EqMainConvReal}.

  Thus, it suffices to prove \eqref{EqMainConv} when $\Lambda$ is precompact \emph{in the lower half plane}, that is, $\ol\Lambda\subset\{\Im\lambda<0\}$. In this case, we can use Proposition~\ref{PropHypBVPPoles}, together with Rouch\'e's Theorem for operator-valued functions, see \cite[Theorem~C.9]{dizzy}; concretely, if $\kappa$ is so small that $\|G_0(\zeta)^{-1}(G_0(\zeta)-G_\kappa(\zeta))\|_{L^2}<1$ for $\zeta\in\pa\Lambda$, then
  \[
    \tr\frac{1}{2\pi i}\oint_{\pa\Lambda} \pa_\lambda G_\kappa(\zeta)\,G_\kappa(\zeta)^{-1}\,d\zeta = \tr\frac{1}{2\pi i}\oint_{\pa\Lambda} \pa_\lambda G_0(\zeta)\,G_0(\zeta)^{-1}\,d\zeta,
  \]
  which is the same as \eqref{EqMainConv}.
\end{proof}

\def\arXiv#1{\href{http://arxiv.org/abs/#1}{arXiv:#1}}

\end{document}